\documentclass[11pt]{article}

\usepackage{hyperref}
\usepackage{setspace}
\usepackage{graphicx}
\usepackage{multirow}
\usepackage{amsmath,amssymb,amsfonts}
\usepackage{amsthm}
\usepackage{mathrsfs}
\usepackage[title]{appendix}
\usepackage{xcolor}
\usepackage{textcomp}
\usepackage{manyfoot}
\usepackage{booktabs}
\usepackage{algorithm}
\usepackage{algorithmicx}
\usepackage{algpseudocode}

\usepackage[a4paper, margin=2.5cm]{geometry}
\usepackage{listings}
\usepackage{comment}
\usepackage{enumitem}
\usepackage{url}
\usepackage{subcaption}
\usepackage{etoolbox}

\DeclareMathOperator{\conv}{Co}

\DeclareMathOperator{\inte}{int}
\DeclareMathOperator{\dom}{dom}
\DeclareMathOperator{\rge}{rge}
\DeclareMathOperator{\graph}{gph}
\DeclareMathOperator{\Fix}{Fix}

\newcommand{\tol}{\mbox{tol}}
\newcommand{\bS}{\mathbf{S}}

\newcommand{\MVI}{\mathbf{M_{\sc vi}}}
\newcommand{\MQVI}{\mathbf{M_{\sc qvi}}}
\newtheorem{theorem}{Theorem}[section]

\newtheorem{proposition}[theorem]{Proposition}
\newtheorem{corollary}[theorem]{Corollary}
\theoremstyle{definition}
\newtheorem{definition}[theorem]{Definition}
\theoremstyle{assumption}
\newtheorem{Assum}[theorem]{Assumption}
\newtheorem{example}[theorem]{Example}
\theoremstyle{remark}
\newtheorem{remark}[theorem]{Remark}
\numberwithin{equation}{section}
\algrenewcommand\algorithmicindent{2em}

\title{Solving Quasi-Variational Inequalities Using the 
Progressive Decoupling of Linkages}
\author{
  Manoel Jardim\thanks{IMPA, Estrada Dona Castorina 110, Rio de Janeiro, RJ 22460-320, Brazil. 
\texttt{manoel.jardim@impa.br}}
  \and
  Claudia Sagastizábal\thanks{IMECC, Unicamp, Rua Sérgio Buarque de Holanda, Campinas, SP 
13083-859, Brazil. \texttt{sagastiz@unicamp.br}}
  \and
  Mikhail Solodov\footnotemark[1] \thanks{\texttt{solodov@impa.br}}
}

\begin{document}
\maketitle
\abstract{Inspired by the progressive decoupling of linkages methodology
for optimization and variational inequalities, 
we propose an algorithm for solving quasi-variational inequalities as a sequence of 
variational inequalities. Our method is shown to converge locally under some 
regularity conditions and globally when such conditions hold throughout the entire domain.
Separately, under other type of assumptions, global convergence with linear rate
 is also established for the class of quasi-variational inequalities
said to have a moving set. 
Advantageous computational performance is shown for large-scale Walrasian equilibrium problems, 
a special case of generalized Nash equilibria, as
well as for some other instances encountered in related literature.}
\vspace{.5cm}
\noindent

\textbf{Keywords:} quasi-variational inequality; 
generalized Nash equilibrium problem; 
progressive decoupling; Spingarn's partial inverse; 
proximal point method.

\section{Introduction}\label{sec1}
The quasi-variational inequality (QVI) is a framework for solving 
a broad spectrum of variational and equilibrium problems. 
Given $F:\mathbb{R}^n\to\mathbb{R}^n$  and a set-valued mapping 
$K:\mathbb{R}^n\rightrightarrows \mathbb{R}^n$, the QVI formulation is
as follows:
\begin{equation}\label{QVI}
\text{find }\bar{x} \in K(\bar{x}), \text{ such that } \langle F(\bar{x}),x-\bar{x}\rangle \geq 0, 
\;\mbox{for all }x \in K(\bar{x}),
\end{equation}
where $\langle \cdot,\cdot\rangle$ stands for the Euclidean inner product in $\mathbb{R}^n$.
The QVI setting is more general than the well-known variational inequalities (VI),
the latter corresponding to the case when
the set-valued mapping $K(\cdot)$ is constant (the same set $K$ for
all $x$).

The first QVI formulation was proposed in
\cite{Bensoussan1973a} and \cite{Bensoussan1973b}
for stochastic impulsive optimal control problems, see also \cite{Bensoussan1975}. 
Since then, the framework has shown itself to be very valuable in economics, for solving
generalized Nash equilibrium problems (GNEP) 
\cite{Harker_1991}, \cite{Aussel_2021}, \cite{Facchinei_2007}, as well as in 
some engineering and transportation applications, see 
\cite{Kravchuk_2007} and \cite{Scrimali_2004} respectively. 
The monograph \cite{2004} discusses 
the mathematical theory of QVIs and VIs in detail.

Much of the literature on QVIs concerns theoretical results, such as the existence and/or uniqueness of solutions. In the context of QVIs in Banach spaces, 
several existence results have been established for specific problems, in a way similar 
 to developments in partial differential equations theory, often together with uniqueness and regularity of solutions. Some examples can be found in \cite{Azevedo_2013}, \cite{Hintermuller_2013}, and 
\cite{Kunze_2000}. 
An existence theorem for a general QVI in Banach spaces is presented
in \cite{Kanzow_2019}.
General theorems establishing existence of solutions for QVIs in Euclidean spaces can be found in \cite{Chan_1982} and \cite{Aussel_2013}. 
There are also existence results for QVIs in specific applications, see for
example \cite{Aussel_2021} and \cite{Donato_2007}. 
However, there are only a few specific theorems addressing uniqueness, such as the results in \cite{Nesterov_2006} and \cite{Dreves_2015}.

More recently, in order to address large-scale structured problems,
the QVI Dantzig-Wolfe decomposition in \cite{jss25} 
demonstrated promising computational efficiency and scalability. 
Nevertheless, in its so-called master subproblem
each Dantzig-Wolfe iteration still solves
a QVI, albeit a simpler one than the original one 
(the second subproblem to be solved at each Dantzig-Wolfe iteration is a VI, often separable).

Our goal in this work is to design an algorithm capable of handling
 large-scale instances of \eqref{QVI} through the iterative solution of VI problems.
To this end, we employ a decomposition approach different from the Dantzig-Wolfe paradigm,  inspired by the progressive decoupling of linkages 
proposed in \cite{Rockafellar_2018} for optimization and VI problems. 
After lifting \eqref{QVI} into a higher dimensional space, a common 
 technique to induce decomposition,
the linkage is given by a subspace constraint. 
In each iteration, progressive decoupling consists of first solving a proximal subproblem in the larger space, 
without the aforementioned constraint, and then performing a 
projection onto the subspace to recover feasibility of the linkage. 
To extend this mechanism to our setting, we first rewrite the QVI in a favorable format. Thanks to 
this reformulation, the subproblem solved by the progressive decoupling approach is just one VI, as desired. The subsequent projection onto the subspace 
has an easy explicit expression.

We expect our new proposal to outperform 
the Dantzig-Wolfe approach \cite{jss25}
when the master QVI subproblem therein is more difficult to solve, which occurs when the problem has more 
constraints and QVI subproblems are not too small. The expectation
is confirmed by the numerical experiments in Section~\ref{sec:num}, where it
is also shown that the situation is reversed for instances 
with fewer constraints.
For computational assessment, we solve Walrasian equilibrium problems formulated as
GNEPs, using their optimality conditions that give rise to a QVI,
 as in \cite{Harker_1991}, \cite{Facchinei_2013}, \cite{QVILIB}.
We compare our method with a direct application of the PATH solver in GAMS \cite{GAMS} and
with the Dantzig-Wolfe approach of \cite{jss25}. Additionally,
we solve several other QVI examples from \cite{QVILIB} and \cite{Deride_2017}
that are useful to
analyze some convergence features of our method, including sensitivity
with respect to the subproblems' proximal parameter.

When dealing with proximal subproblems,
one difficulty which is inherent in QVIs is the lack of monotonicity-type properties of
the associated mapping. We briefly discuss this next.
 Let $K(x)$ be convex for each $x$.
For a convex set $D$, the normal cone to $D$ at $x$ is the set 
$N_{D}(x)=\{v\in\mathbb{R}^n: \langle v,y-x\rangle \leq 0, \mbox{for all }y \in D\}$
when $x\in D$, and $N_D(x)=\emptyset$ otherwise.
 Then the QVI \eqref{QVI} is equivalent to finding $\bar{x}$ solving the
generalized equation
\begin{equation}\label{QVI2}
0 \in F(\bar{x})+N_{K(\bar{x})}(\bar{x}).
\end{equation}
 Because $K$ in a QVI is not a constant set, but varies with $x$, the
mapping
$N_{K(\cdot)}(\cdot): \mathbb{R}^n \rightrightarrows \mathbb{R}^n$
 may lack properties suitable for a direct application of
proximal point and related techniques, such as the
progressive decoupling method. For this reason, special attention is
required.  Following
\cite{Patriksson_2003} we refer to such a mapping as the putative normal cone. 
In particular,
this mapping is not monotone, even in very simple examples. 
Moreover, even weaker properties like local monotonicity around a solution
or elicited monotonicity, do not appear to be amenable to the QVI case
(a further discussion is given in Section~\ref{sec:reg}; see in particular
Example~\ref{exeli} therein). 
 For this reason, part of our analysis relies on local results. 
We shall prove local convergence, under suitable
regularity conditions and
global convergence when the conditions hold throughout the entire domain. This analysis invokes
variants of the progressive decoupling method proposed in \cite{evens2025}. Note that
even though the method proposed in~\cite{Rockafellar_2018} allows elicitation of monotonicity, 
the progressive decoupling+ from~\cite{evens2025} covers more general cases and 
provides both local and global convergence results under a calmness condition.
In addition, we establish global convergence with linear rate for a specific instance 
of \eqref{QVI}, namely the moving set case, by direct analysis not
involving the theory of \cite{evens2025}.

The rest of the paper is organized as follows. In Section~\ref{sec:alg} and Section~\ref{sec:prog+}, we develop and
formulate the algorithm. In Section~\ref{sec:reg}, 
we introduce the regularity conditions necessary for the convergence analysis. 
Section~\ref{sec:global} and Section~\ref{sec:loc} provide global and local 
convergence analysis, respectively.
The final Section~\ref{sec:num} is devoted to the numerical assessment.

We introduce the following notation.
Denoting the Euclidean norm by $\|\cdot\|$, the closed unit ball in 
$\mathbb{R}^n$ is $\mathbb{B}^n=\{v \in \mathbb{R}^n: \|v\|\leq 1\}$.
For a vector subspace ${\bS}\subset \mathbb{R}^n$, its orthogonal subspace is ${\bS}^\perp=\{u \in \mathbb{R}^n: \langle u,v \rangle = 0, \mbox{for all }v \in{\bS}\}$.
It then holds that
 $\mathbb{R}^n=\bS\oplus {\bS}^\perp=\{v \in \mathbb{R}^n: v=v_{\bS}+v_{{\bS}^\perp}, v_{\bS} \in{\bS}, v_{{\bS}^\perp} \in {\bS}^\perp\}$ is the direct sum decomposition of $\mathbb{R}^n$, where
 $v_{\bS}=\mathbf{P}_{\bS}(v)$ and $v_{{\bS}^\perp}=\mathbf{P}_{{\bS}^\perp}(v)$ are the projections of $v$ onto
 $\bS$ and ${\bS}^\perp$, respectively. 

For a set-valued mapping $\Phi :\mathbb{R}^n\rightrightarrows \mathbb{R}^n$, its domain is given by
$\dom \Phi =\{x\in \mathbb{R}^n : \Phi (x)\neq\emptyset\}$, its graph is
$\graph \Phi =\{(x,u):x \in \dom \Phi , u \in \Phi (x)\}$, its range is 
$\rge(\Phi)=\{u\in \mathbb{R}^n: u \in \Phi(x)\text{ for some } x \in \dom \Phi\}$,
and its set of fixed points is given by
$\Fix (\Phi) = \{x : x\in\Phi (x)\}$. 
The mapping $\Phi$ is $\sigma$-strongly monotone if there exists $\sigma >0$
 such that $\langle v-u,y-x \rangle \geq \sigma \|y-x\|^2, \mbox{for all }(y,v), (x,u) \in \graph \Phi$.  
The mapping $\Phi$ is monotone if the inequality holds with $\sigma =0$.
A monotone set-valued mapping 
is maximally monotone if its graph 
is not properly contained in the graph of any other monotone mapping.

A set-valued mapping $T:\mathbb{R}^n\rightrightarrows\mathbb{R}^n$ is 
$e$-elicitable monotone with respect to the subspace 
$\mathbf{S}$ if $T+eP_{\mathbf{S}^\perp}$ is maximally monotone for some $e\geq0$.
Given a proximal parameter $r>0$, the resolvent of $T$ with parameter $r$ is defined as
$\mathcal{J}_{\frac1r T}:=(I+\frac1r T)^{-1}
:\mathbb{R}^n\rightrightarrows\mathbb{R}^n$, a possibly set-valued mapping
if $T$ is non-monotone.

A set-valued mapping $K:\mathbb{R}^n\rightrightarrows \mathbb{R}^n$ is inner semicontinuous 
at $\bar{u}\in \dom K$ if for all $u^k\to \bar{u}$ and $\bar{v} \in K(\bar{u})$, 
there exists a sequence $v^k\to \bar{v}$ with $v^k \in K(u^k)$. 
In our QVI~\ref{QVI}, the set-valued mapping $K(\cdot)$ is assumed to be closed and convex-valued, 
meaning that the sets $K(x)$ are closed and convex for each 
$x \in \dom K$. Moreover, $K(\cdot)$
is assumed to be inner semicontinuous and to have a closed graph. 
Finally, the variational inequality obtained when $K$ in \eqref{QVI} 
is a fixed set
is denoted by VI$(F,K)$.

\section{QVI Progressive Decoupling}\label{sec:alg}

A useful mechanism to induce decomposition in a 
problem is to make copies of variables appearing in coupling constraints, 
in a manner that separable structures can be exploited.
The mechanism is defined in a decision space of larger dimensions,
but in compensation each iteration yields separate subproblems,
generally simpler to solve, possibly in parallel. 
Decisions in the original space are obtained by forcing
the copies to be equal, and the iterative process continues.

\subsection{The considered setting}\label{ss:set}
We denote with bold font all vectors, mappings, and sets in $\mathbb R^{2n}$, the space of copies. 
 We introduce the subspace  
{\footnotesize $\bS=\left\{\mathbf{x}=\begin{bmatrix}x_1\\x_2\end{bmatrix}\in\mathbb{R}^n\times\mathbb{R}^n: 
x_1=x_2\right\},$} and its orthogonal complement 
{\footnotesize ${\bS}^\perp=\left\{\mathbf{y}=\begin{bmatrix}y_1\\y_2\end{bmatrix}\in \mathbb{R}^n \times \mathbb{R}^n: y_1=-y_2\right\}$}.
The associated projection operators are
\begin{equation}
\label{projs2}
\mathbf{P}_{\bS}\left(\mathbf{x}\right)=
\frac{1}{2}\begin{bmatrix}x_1+x_2 \\  x_1+x_2 \end{bmatrix}\mbox{ and }
\mathbf{P}_{{\bS}^\perp}\left(\mathbf{x}\right)=
\frac{1}{2}\begin{bmatrix} x_1-x_2 \\ x_2-x_1\end{bmatrix}
\mbox{ for any } \mathbf{x}\in \mathbb{R}^{2n}\,.
\end{equation}
The QVI formulation \eqref{QVI2} in the space of
copies, is given by
\begin{align}
 \left(\mathbf{x}, \mathbf{y}\right) 
\in {\bS}\times {\bS}^\perp, &\mbox{ with }&\hspace{-5em} \mathbf{y} \in \mathbf{T}\left(\mathbf{x}\right)=(F(x_1)+N_{K(x_2)}(x_1))\times \{0\},
\label{decoupling}\\
&\mbox{ and }&\hspace{-19em}x_1 \in K(x_2)\,.\hspace{10em} \label{fp} 
\end{align}

This is Spingarn's partial inverse problem, where the goal is to find a pair in the graph of the operator such that the point and its image 
respectively lie in a subspace and its orthogonal complement. 
The transformation was proposed in~\cite{Spingarn_1983} for convex
separable programming
problems, and the progressive decoupling method for VIs
\cite{Rockafellar_2018} is based on this type of construction.

To analyze the framework in QVI setting, we 
begin by examining one iteration of the
progressive decoupling of linkages 
when applied to \eqref{decoupling}--\eqref{fp}. 
At iteration $k$,
given $(\mathbf{x}^k,\mathbf{y}^k) \in {\bS}\times {\bS}^\perp$ and 
a proximal parameter $r_k>0$, the
next iterate $(\mathbf{x}^{k+1},\mathbf{y}^{k+1})\in {\bS}\times {\bS}^\perp$ is generated in two steps, by solving a proximal subproblem and projecting. 
Specifically,
\begin{eqnarray}
\text{first,}& \hat{\mathbf{x}}^k \text{ is a solution to } &\mathbf{0} \in \mathbf{T}(\mathbf{x})-\mathbf{y}^k+r_k (\mathbf{x}-\mathbf{x}^k), 
\label{ite}\\
\mbox{then,}&\mbox{\eqref{projs2} gives}\hspace{1.5cm}&
\mathbf{x}^{k+1}=\mathbf{P}_{\bS}(\hat{\mathbf{x}}^k) \mbox{ and } \mathbf{y}^{k+1}=\mathbf{y}^k-
r_k\mathbf{P}_{{\bS}^\perp}(\hat{\mathbf{x}}^k).  
\label{projs1}
\end{eqnarray}
Although the original progressive decoupling method is formulated for a fixed proximal parameter, 
we write the iteration with a variable parameter for greater generality. 
We restrict to the case of a constant parameter whenever required for the convergence analysis, 
except for the global convergence result in the moving-set setting (where we show that
variable parameters are admissible indeed).

Two important comments are now in order. To begin with,
because a solution $\hat{\mathbf{x}}^k$ to \eqref{ite} satisfies the inclusion
$ \mathbf{0} \in \mathbf{T}(\hat{\mathbf{x}}^k)-\mathbf{y}^k+r_k(\hat{\mathbf{x}}^k-\mathbf{x}^k)$,
it holds that
\[
\begin{bmatrix}0 \\ 0 \end{bmatrix} \in \begin{bmatrix} F(\hat{x}^k_1)+
r_k\hat{x}^k_1-y_1^k-r_k x^k_1+N_{K(\hat{x}^k_2)}(\hat{x}^k_1) \\ 
r_k\hat{x}^k_2-y_2^k-r_k x_2^k\end{bmatrix}\,.         
\]
Thus, at iteration $k$, the two components $\hat x_1^k$ and $\hat x_2^k$ in \eqref{ite} are such that
$\hat{x}_2^k=x_2^k+\frac{1}{r_k}y_2^k$,  while
$\hat{x}_1^k$ solves VI $(F+r_k I-y_1^k-r_k x_1^k,K^k)$, where we
defined the set $K^k=K(\hat{x}_2^k).$

Second, by the projection formula in \eqref{projs1}, the next iterates
are
\begin{align}
\mathbf{x}^{k+1}=\frac{1}{2}\begin{bmatrix}\hat{x}_1^k+\hat{x}_2^k \\ 
\hat{x}^k_1+\hat{x}_2^k\end{bmatrix} &\implies 
x_1^{k+1}=x_2^{k+1}=\frac{1}{2}(\hat{x}^k_1+x^k_2+\frac{1}{r_k}y_2^k), \notag \\
\mathbf{y}^{k+1}=\mathbf{y}^k-\frac{r_k}{2}\begin{bmatrix}\hat{x}^k_1-\hat{x}_2^k\\ 
\hat{x}_2^k- \hat{x}_1^k\end{bmatrix}&\implies 
\begin{matrix}y_1^{k+1}=y_1^k-\frac{r_k}{2}(\hat{x}_1^k-x_2^k-\frac{1}{r_k}y_2^k), \\ 
 y_2^{k+1}=y_2^k-\frac{r_k}{2}(x_2^k+\frac{1}{r_k}y_2^k-\hat{x}_1^k).\end{matrix} \notag
\end{align}
Since $\mathbf{x}^k \in\bS$ and $\mathbf{y}^k \in\bS^{\perp}$, the respective
components satisfy
$x_1^k=x_2^k$ and $y_2^k=-y_1^k$. Thus, renaming $x^k_1=x_2^k=x^k,$ and 
$y_1^k=y^k, y_2^k=-y^k$ yields the expression
$\hat{x}^k_2=x^k-\frac{1}{r_k}y^k$. 
We can also rename $\hat{x}_1^k=\hat{x}^k$. 

Altogether, the proximal subproblem of the iteration consists in finding
\[ \hat{x}^k \mbox{ solving the VI}\Bigl(F+r_k I-y^k-r_k x^k,K(x^k-\frac{1}{r_k}y^k)\Bigr)\,.\]

\subsection{A first algorithmic setting}

Since we have characterized $\hat x^k$ as a solution to a VI with
operator $F+r_k I-y^k-r_k x^k$ and set 
$K^k=K(x^k-\frac{1}{r_k}y^k)$, the subproblems are well defined without 
further 
assumptions/considerations, at least locally. 
To see this, note that as
$K$ is inner semicontinuous in $\bar{x}$, then $\bar{x} \in \inte \dom K$. 
Therefore, $K^k$ will be non-empty whenever $(x^k,y^k)$ is sufficiently close to $(\bar{x},0)$. 
 If, in addition, $F$ were monotone, the 
 VI mapping $F(\cdot)+r_kI-y^k-r_kx^k$ would be strongly monotone. 
In this case,
by Proposition~2.3.3-(b) in \cite{2004}, 
the corresponding $k$th VI subproblem has 
the unique solution $\hat x^k$.

\begin{center}
\rule{\linewidth}{0.8pt}
\refstepcounter{algorithm}
\textbf{Algorithm \thealgorithm} Progressive decoupling for QVI~\ref{QVI}
\label{algo1}
\rule{\linewidth}{0.8pt}
\begin{algorithmic}
\Require $x^0,y^0 \in \mathbb{R}^n, \tol \ge 0, \varepsilon^0 > \tol.$
\Ensure Convergence to solutions of the QVI.
\State Set $k \leftarrow 0$.
\While{$\varepsilon^k > \tol$} 
\State 
 Choose $r_k >0$ and set $K^k=K(x^k-\frac{1}{r_k}y^k)$;
\State
find $\hat{x}^k \text{ as a solution to VI}\,(F+r_kI-y^k-r_kx^k,K^k)$; 
\State
set $x^{k+1}=\frac{1}{2}(\hat{x}^k+x^k-\frac{1}{r_k}y^k)$;
\State
set $y^{k+1}=y^k-\frac{r_k}{2}(\hat{x}^k-x^k+\frac{1}{r_k}y^k)$.
\State {\sc stopping criterion}: 
Compute $\varepsilon^k=\|x^{k+1}-x^k\|+\frac{1}{r_k}\|y^{k+1}-y^k\|$;
\State {\sc update}: 
$k \leftarrow k+1$;
\EndWhile
\end{algorithmic}
\rule{\linewidth}{0.8pt}
\end{center}

Note that the iterative procedure given in Algorithm~\ref{algo1}
 solves one variational inequality in $\mathbb{R}^n$ per iteration, the remaining
calculations being simple algebraic operations. 
The stopping criterion is based on the
convergence of the corresponding sequences, whose analysis is provided in 
Sections~\ref{sec:global} and \ref{sec:loc}. For now,
we provide some insights as to why,
if the method converges in some sense, then a solution to \eqref{QVI}
is obtained  (of course, approximate solutions for a positive tolerance $\tol >0$).

\begin{proposition}[Meaning of the stopping criterion]
Suppose that in QVI~\ref{QVI} the operator
$F:\mathbb{R}^n\to \mathbb{R}^n$ is continuous and 
$K:\mathbb{R}^n\rightrightarrows \mathbb{R}^n$ is an inner semicontinuous set-valued mapping 
with closed graph. 

Suppose, in addition, the proximal parameters 
in Algorithm \ref{algo1} are chosen so that 
$0<\check{r}\leq r_k\leq \hat{r}< \infty$ for all $k$.

If the process generates a sequence
$\{y^k\}\to 0$, then every accumulation point 
$\bar{x}$ of the sequence $\{x^k\}$ is a solution of the QVI \eqref{QVI}. 
Moreover, whenever $\{x^k\} \to \bar{x} $ the sequence $\{y^k\}\to 0,$ and $\bar{x}$ solves the QVI \eqref{QVI}. 
\end{proposition}
\begin{proof}
Suppose that $\{y^k\}\to 0$ and consider $k\to\infty$. As 
$y^{k+1}=-\frac{r_k}{2}(\hat{x}^k-x^k-\frac{1}{r_k}y^k)$,
our conditions on $r_k$ imply that
$\hat{x}^k-x^k \to 0$.

For $\bar{x} $ an accumulation point of $\{x^k\}$, consider the subsequence
$\{x^{k_j}\}\to \bar x$ as $j\to\infty$. Then it also holds that
$\{\hat{x}^{k_j}\} \to \bar{x}$ as $j\to\infty$.
 Since for all $k$ the points $\hat{x}^k \in K(x^k-\frac{1}{r_k}y^k)$ and $\graph K$ is closed,
passing to the limit along the subsequence $\{k_j\}$ as $j\to\infty$,
we conclude that $\bar{x} \in K(\bar{x}).$ 

Next, take an arbitrary $u \in K(\bar{x})$. By        
the inner semicontinuity of $K(\cdot)$, together with the relation
$x^{k_j}-\frac{1}{r_{k_j}}y^{k_j} \to \bar{x}$, there exists some subsequence 
$\{u^{k_j}\}\to u$ such that $u^{k_j} \in K^{k_j}= K(x^{k_j}-\frac{1}{r_{k_j}}y^{k_j})$
for all $j$. 
Since $\hat{x}^{k_j}$ solves the corresponding VI, it holds that
\[\left< F(\hat{x}^{k_j})+r_{k_j}\hat{x}^{k_j}-y^{k_j}-r_{k_j}x^{k_j}, u^{k_j}-\hat{x}^{k_j}\right> \geq 0 .\]
Passing onto the limit as $j \to \infty$ yields that
 $\langle F(\bar{x}),u-\bar{x}\rangle \geq 0$, a relation that holds for any 
$u \in K(\bar{x})$. Hence,
 $\bar{x}$ solves the QVI~\eqref{QVI}, as claimed.

Finally, suppose that $x^k\to \bar{x}$. The result follows, because $x^{k+1}-x^k\to 0$ and $$y^{k+1}=-\frac{r_k}{2}(\hat{x}^k-x^k-\frac{1}{r_k}y^k)=-\frac{r_k}{2}(\hat{x}^k+x^k-\frac{1}{r_k}y^k-2x^k)=-r_k(x^{k+1}-x^k)\to 0.$$
That this relation implies
that $\bar{x}$ is a solution to \eqref{QVI} was established above.
\end{proof}

At this stage, it is convenient to recall the relation between the proximal subproblem in Algorithm~\ref{algo1}, written in the original space ($\mathbb R^n$), and its equivalent expression 
involving the resolvent of the operator in \eqref{decoupling}, defined
in the lifted space of copies, that is:
\[\hat{x}^k \text{ solves the VI}\,(F+r_kI-y^k-r_kx^k,K^k)\iff
\hat{\mathbf x}^k= \begin{bmatrix}\hat{x}^k\\ x^k-\frac{1}{r_k}y^k\end{bmatrix} 
\in\mathcal{J}_{\frac1{r_k}\mathbf T}\Bigl(\mathbf x^k+\frac{1}{r_k}\mathbf y^k\Bigr)
\,.\]
Together with \eqref{projs2}, 
we see that, when written in the lifted space of copies, 
the $k$th iteration of Algorithm\ref{algo1} performs the following steps
\begin{equation}\label{lifted-it}
\begin{cases}\hat{\mathbf{x}}^k \in \mathcal{J}_{\frac1{r_k}\mathbf{T}}(\mathbf{x}^k+\frac{1}{r_k}\mathbf{y}^k)\\
 \mathbf{x}^{k+1}=\mathbf{P}_\mathbf{S}(\hat{\mathbf{x}}^k) \\ \mathbf{y}^{k+1}
=\mathbf{y}^k-r_k\mathbf{P}_{\mathbf{S}^\perp}(\hat{\mathbf{x}}^k)\,. \end{cases}\end{equation}
This rewriting will be useful in the following section to introduce a second algorithmic approach for solving \eqref{QVI},
 when written in its equivalent form \eqref{decoupling}-\eqref{fp}.

\section{QVI Progressive Decoupling+}\label{sec:prog+}

The progressive decoupling method \cite{Rockafellar_2018}
 handles non-monotonicity of the mapping $\mathbf T$ 
through elicitation within the linkage subspace. 
Since that setting does not accommodate the lack of monotonicity of the 
putative cone (see Section~\ref{sec:reg} and Example~\ref{exeli}),
we consider instead the weaker concept of semimonotonicity,
introduced in \cite{evens2025}.

The modified variant of the progressive decoupling
approach presented in \cite{evens2025} can be employed when elicitation is not applicable.
The resulting method, called progressive decoupling+, incorporates 
in the original framework
separate relaxation parameters, $\lambda_x$ and $\lambda_y$ when updating the iterates in \eqref{lifted-it}.
As long as the proximal parameter is fixed ($r_k\equiv r$ for all iterations),
the mechanism allows for a controlled form of 
non-monotonicity along the subspace directions and their complements. 

The approach in \cite{evens2025}
applied to our QVI problem writes down as follows.
Given $(\mathbf{x}^k,\mathbf{y}^k)\in \mathbf{S}\times\mathbf{S}^\perp,$ a proximal parameter $r>0$, and relaxation parameters $\lambda_x,\lambda_y$, 
the iteration of the progressive decoupling+ defines
\begin{equation*}
\begin{cases}\hat{\mathbf{x}}^k \in \mathcal{J}_{\frac1r\mathbf{T}}(\mathbf{x}^k+\frac{1}{r}\mathbf{y}^k)\\
 \mathbf{x}^{k+1}=(1-\lambda_x)\mathbf{x}^k+\lambda_x\mathbf{P}_\mathbf{S}(\hat{\mathbf{x}}^k) \\ \mathbf{y}^{k+1}
=\mathbf{y}^k-\lambda_yr\mathbf{P}_{\mathbf{S}^\perp}(\hat{\mathbf{x}}^k).
\end{cases}\end{equation*}
We see that taking
$\lambda_x=\lambda_y=1$ gives the progressive decoupling steps of Algorithm~\ref{algo1}, 
written in the form \eqref{lifted-it} and fixing $r_k=r$
(in the progressive decoupling variant with elicitation the respective values
of the parameters are $\lambda_x=1$ and $\lambda_y=1+\frac{e}{r}$).

Algorithm~\ref{algo2} stated below results from \cite{evens2025},
after lifting \eqref{QVI} to the space of copies
as in Section~\ref{ss:set}, using the operator $\mathbf T$ 
given in \eqref{decoupling}.

\begin{center}
\rule{\linewidth}{0.8pt}
\refstepcounter{algorithm}
\textbf{Algorithm \thealgorithm} Progressive decoupling+ for QVI~\ref{QVI}
\label{algo2}
\rule{\linewidth}{0.8pt}
\begin{algorithmic}
\Require $x^0,y^0 \in \mathbb{R}^n, r>0, \lambda_x > 0, \lambda_y >0, \tol \ge 0,
\varepsilon^0 > \tol.$
\Ensure Convergence to solutions of the QVI.
\State Set $k \leftarrow 0$.
\While{$\varepsilon^k > \tol$} 
\State 
Set $K^k=K(x^k-\frac{1}{r}y^k)$;
\State
find $\hat{x}^k \text{ as a solution to VI}\,(F+rI-y^k-rx^k,K^k)$; 
\State
set $x^{k+1}=(1-\lambda_x)x^k+\lambda_x\left(\frac{1}{2}(\hat{x}^k+x^k-\frac{1}{r}y^k)\right)$;
\State
set $y^{k+1}=y^k-\lambda_y\left(\frac{r}{2}(\hat{x}^k-x^k+\frac{1}{r}y^k)\right)$.
\State {\sc stopping criterion}: 
Compute $\varepsilon^k=\|x^{k+1}-x^k\|+\frac{1}{r}\|y^{k+1}-y^k\|$;
\State {\sc update}: 
$k \leftarrow k+1$;
\EndWhile
\end{algorithmic}
\rule{\linewidth}{0.8pt}
\end{center}

For ensuring convergence, the parameters $r$ and $\lambda_x, \lambda_y$ have to satisfy certain conditions, which will appear in due course, after
the study of semimonotonicity in our QVI setting in the next section.

\section{On semimonotonicity in QVI setting}\label{sec:reg}

Since the progressive decoupling of linkages is an application of the proximal 
point method (PPM) \cite{Rockafellar_2018}, its convergence properties 
follow from those of that 
method under appropriate assumptions. Convergence of PPM depends on monotonicity-related 
properties (perhaps local monotonicity, perhaps elicitation), 
which are not easily achievable in the presence of 
putative normal cones. 

\subsection{Regularity Assumptions}
The structure of the mapping $\mathbf{T}$ 
in \eqref{decoupling} is not amenable to monotonicity, but the assumptions required by the 
progressive decoupling+ are less demanding and turn out to be reasonable
for quasi-variational inequalities. 

We illustrate the situation with the following simple example, in which we have neither monotonicity
--nor local monotonicity around the solution
--nor elicitation for any parameter.
However, as will be shown later on in Example~\ref{exglo}, 
we have semimonotonicity at the solution point 
in the sense of \cite{evens2025}, stated in Definition~\ref{semimon}.

\begin{example}\label{exeli}\em
Consider QVI \eqref{QVI} defined with $n=2$, the identity operator
 $F(x)=x$ and the set-valued mapping
$K(x)=\{tx: t \in \mathbb{R}\}\subset \mathbb{R}^2$. 

It is easy to see that
\[N_{K(x)}(x)=[x]^\perp:=\{h\in \mathbb{R}^2: \langle h, x\rangle =0\}.\]
Then the unique solution to  QVI \eqref{QVI2}, i.e., $0 \in F(x)+N_{K(x)}(x)$, 
equivalent to \eqref{QVI}, is
$\bar{x}=0$.

To derive the operator in the lifted space, first note that
given $x_1,x_2 \in \mathbb{R}^2$, 
$$N_{K(x_2)}(x_1)=
\begin{cases}[x_2]^\perp:=
\{h \in \mathbb{R}^2: \langle h,x_2 \rangle =0\}, \text{ when } x_1 \in K(x_2), \\
 \emptyset, \text{ when } x_1 \notin K(x_2).
\end{cases}$$
It follows that
the mapping $\mathbf{T}$ from~\eqref{decoupling} is given by          
\[\mathbf{T}:\mathbb{R}^4 \rightrightarrows \mathbb{R}^4, \mathbf{T}
\left(\begin{bmatrix}x_1\\x_2\end{bmatrix}\right)=
\begin{cases}(x_1+[x_2]^\perp)\times\{0\}, \text{ if } x_1 \text{ and } x_2 \text{ are parallel,}\\ 
\emptyset, \text{ otherwise.} \end{cases}\]
Furthermore, keeping \eqref{projs2} in mind, for any 
$\varepsilon>0,\delta>0,$ 
\[
 \mathbf T\left(\begin{bmatrix}\varepsilon \\ 0 \\ -\varepsilon \\ 0 \end{bmatrix}\right)
=\left(\begin{bmatrix}\varepsilon\\0\end{bmatrix}+\left\{\begin{bmatrix}0\\t\end{bmatrix}\in \mathbb{R}^2, t\in \mathbb{R}\right\}\right) \times \left\{\begin{bmatrix}0\\0\end{bmatrix}\right\} 
=\left\{\begin{bmatrix}\varepsilon\\ t \\ 0 \\ 0\end{bmatrix}\in \mathbb{R}^4, t \in \mathbb{R}\right\}, 
\]
while, similarly,
\[
\mathbf T\left(\begin{bmatrix}0 \\ \delta \\ 0 \\ -\delta \end{bmatrix}\right) 
= \left\{\begin{bmatrix} t \\ \delta \\ 0 \\0 \end{bmatrix}\in\mathbb{R}^4: t \in \mathbb{R}\right\}.
\]
Therefore, 
\begin{align}
\mathbf{P}_{{\bS}^\perp}\left(\begin{bmatrix} \varepsilon \\ 0 \\ -\varepsilon \\ 0\end{bmatrix}\right) 
= \begin{bmatrix} \varepsilon \\0 \\ -\varepsilon \\ 0  \end{bmatrix}, \quad
 \mathbf{P}_{{\bS}^\perp}\left(\begin{bmatrix} 0 \\ \delta \\ 0 \\ -\delta \end{bmatrix}\right) 
= \begin{bmatrix} 0 \\ \delta \\ 0 \\ -\delta \end{bmatrix},\notag
\end{align}
and, hence,
\[\left(\begin{bmatrix}0 \\ \delta \\ 0 \\ -\delta \end{bmatrix}, 
\begin{bmatrix}t_2 \\ \delta \\ 0 \\ 0 \end{bmatrix}+e
\begin{bmatrix}0 \\ \delta \\ 0 \\ -\delta\end{bmatrix} \right), \left(
\begin{bmatrix}\varepsilon \\ 0 \\ -\varepsilon \\ 0\end{bmatrix}, 
\begin{bmatrix}\varepsilon\\t_1\\0\\0\end{bmatrix}+e
\begin{bmatrix}\varepsilon\\0\\-\varepsilon\\0\end{bmatrix}\right) \in \graph(\mathbf{T}+e\mathbf{P}_{{\bS}^\perp}), \text{ for } t_1, t_2 \in \mathbb{R}.\]
In particular, given any elicitation parameter $e>0$, 
\begin{align}
\left\langle 
\begin{bmatrix}t_2 \\ (1+e)\delta \\ 0 \\ -e\delta \end{bmatrix} - 
\begin{bmatrix} (1+e)\varepsilon \\ t_1 \\ -\varepsilon e \\ 0\end{bmatrix}, 
\begin{bmatrix}0 \\ \delta \\ 0 \\ -\delta \end{bmatrix}-
\begin{bmatrix}\varepsilon \\ 0 \\ -\varepsilon \\ 0\end{bmatrix} \right\rangle = (1+2e)(\delta^2+\varepsilon^2)-t_1\delta-t_2\varepsilon<0, \\ \text{ whenever } t_1\delta+t_2\varepsilon>(1+2e)(\delta^2+\varepsilon^2).\notag
\end{align}
We conclude that for this example, $\mathbf{T}+e\mathbf{P}_{{\bS}^\perp}$ cannot be monotone 
for any $e>0$. 
Thus, $\mathbf{T}$ is not an $e$-elicitable mapping and, moreover, taking 
$e=0$, shows that $\mathbf{T}$ is not monotone either. 
Monotonicity cannot be ensured even locally around the solution
$(0,0)\in \mathbb{R}^4\times \mathbb{R}^4$, because we can choose $\delta,\varepsilon$ above
 as small as 
we wish, without altering the conclusion. On the other hand, the semimonotonicity condition 
required by progressive decoupling+ is readily satisfied for this example, as
will be illustrated in Example~\ref{exglo} below.\qed
\end{example}

As the mapping $\mathbf{T}:\mathbb{R}^{2n}\rightrightarrows \mathbb{R}^{2n}$ 
cannot be expected to be
 maximally monotone, locally monotone, or 
elicitable in the sense of \cite{Rockafellar_2018},
the results on convergence of PPM cannot be applied to
 progressive decoupling in our setting.
Instead, we shall employ the framework of \cite{evens2025}, which is 
based on the notion of semimonotonicity.

\begin{definition}[Semimonotonicity]\label{semimon}
Given a mapping $\Phi:\mathbb{R}^m\rightrightarrows \mathbb{R}^m$, 
a subspace $\mathcal{V} \subset \mathbb{R}^m$ and scalars $\mu^\perp\,,\mu>0$, 
the mapping $\Phi$ is $(\mu^\perp,\mu)$- 
semimonotone at $(x^*,y^*)\in \graph \Phi$ on an open set $\mathcal{U}\ni  (x^*,y^*)$ if, 
for all $(x,y) \in \graph \Phi \cap \mathcal{U}$ the following holds:
$$\langle y-y^*,x-x^*\rangle + \mu^\perp\|\mathbf P_{\mathcal{V}^\perp}(x)-\mathbf P_{\mathcal{V}^\perp}(x^*)\|^2+\mu\|\mathbf P_\mathcal{V}(y)-\mathbf P_\mathcal{V}(y^*)\|^2 \geq 0\,.$$
\end{definition}

For our mapping $\mathbf{T}$ in~\eqref{decoupling} and the subspace $\mathbf S$,
this property  will be satisfied 
at $(\mathbf{\bar{x}},\mathbf{0})$ if there exist $\delta_x,\delta_y>0$ such that
\begin{equation}\label{ineqQ}
\langle y,x_1-\bar{x}\rangle+\frac{\mu^\perp}{2}\|x_1-x_2\|^2+\frac{\mu}{2}\|y\|^2 \geq 0 \;
\begin{array}{l}
\mbox{for all }x_1,x_2\in \bar x+\delta_x\mathbb B^n\mbox{ and }
y\in \delta_y\mathbb B^n\\
 \text{such that } y \in F(x_1)+N_{K(x_2)}(x_1).\end{array}
\end{equation}

Recalling that Algorithm~\ref{algo1} is a particular case  of Algorithm~\ref{algo2}, 
to show convergence for both of our QVI progressive decoupling methods, 
we  next state conditions ensuring that the property \eqref{ineqQ} of
semimonotonicity holds.

To this end, given parameters $p,q \in \mathbb{R}^n$,
we introduce the following parametric VI and QVI solution mappings:
\begin{equation}\label{paraVI} 
\begin{array}{lll}
x \in \mathbf{\MVI}(p, q) &\iff& 0 \in p+F(x)+N_{K(q)}(x), \\
x \in \mathbf{\MQVI}(p,q) &\iff& 0 \in p+F(x)+N_{K(x+q)}(x).
\end{array}  
\end{equation}
Stability properties of parameterized VIs and QVIs have been studied in \cite{Robinson_1980}, \cite{Mordukhovich_1994}, \cite{Outrata_1996}, \cite{Mordukhovich_2007}. Some results are based on metric subregularity theory, which is equivalent to a calmness property for the solution mapping.

We now introduce the regularity conditions for our mappings. For the convergence analysis, it suffices that either of the two assumptions below holds.  
The difference between the local and global analyses is that the former requires subregularity only locally, while the latter requires the
property to hold on the entire domain.

\begin{Assum}[$\MVI$-calmness]\label{m1}
Let $\bar{x}$ be a solution of QVI~\eqref{QVI}, meaning that $\bar{x} \in \mathbf{\MVI}(0,\bar{x})$. The mapping $\mathbf{\MVI}$ is locally single-valued and calm
at $((0,\bar{x}),\bar{x})$: there exist neighborhoods
$V_p \ni 0, V_q \ni \bar{x}$ and constants $L_p, L_q \geq 0$
with $L_q < 1/2$ such that $\mathbf{\MVI}$ is single-valued in $V_p\times V_q$ and
$$
  \|\mathbf{\MVI}(p,q)-\mathbf{\MVI}(0,\bar{x})\|
  \leq L_p\|p\|+L_q\|q-\bar{x}\|, \text{ for all } (p,q) \in V_p\times V_q.
$$
\end{Assum}

\begin{Assum}[$\MQVI$-calmness]\label{m2}
Let $\bar{x}$ be a solution of QVI~\eqref{QVI}, meaning that $\bar{x} \in \mathbf{\MQVI}(0,0)$. The mapping $\mathbf{\MQVI}$ is locally single-valued and calm
at $((0,0),\bar{x})$: there exist neighborhoods
$V_p \ni 0, V_q \ni 0$ and constants $L_p, L_q \geq 0$
with $L_q < 1$ such that $\mathbf{\MQVI}$ is single-valued in $V_p\times V_q$ and
$$
  \|\mathbf{\MQVI}(p,q)-\mathbf{\MQVI}(0,0)\|
  \leq L_p\|p\|+L_q\|q\|, \text{ for all } (p,q) \in V_p\times V_q.
$$
\end{Assum}

\begin{remark}
\begin{enumerate}[label=(\roman*)]
\item
When $F$ is strongly monotone, the mapping $\mathbf{\MVI}$ is always
single-valued, since uniqueness of the solution of the VI is ensured by
Theorem 2.3.3-(b) in \cite{2004}. 
\item
Assumption~\ref{m1} does not require uniqueness or local uniqueness of solutions to QVI~\eqref{QVI}.
 Indeed, $\mathbf{\MVI}$ can be single-valued even when QVI~\eqref{QVI} has multiple solutions. 
On the other hand, single-valuedness of $\mathbf{\MQVI}$, locally or globally, 
requires local or global uniqueness of the solution of QVI~\eqref{QVI}.
\item
The calmness property can be analyzed using
variational analysis tools, such as coderivative criteria for metric
subregularity \cite{Dontchev_2004, Dontchev_2021}. Estimating the calmness moduli $L_p$ and $L_q$
is possible via coderivatives, but is not straightforward in general.
\item
Assumptions~\ref{m1} and~\ref{m2} are independent (neither implies the other), 
as shown by Examples~\ref{m1notm2} and~\ref{m2notm1} below.
\end{enumerate}
\end{remark}

Verifying when the assumptions above hold is not a simple task,
 but we can expect them to be satisfied 
for nontrivial classes of QVIs, as illustrated by the following example.           

\begin{example}\em
Consider affine QVI \eqref{QVI}, given by                              
\[F(x)=Ax+b, \quad K(x)=\{z\in \mathbb{R}^n: Cx+Dz=d\},\]
where $A\in\mathbb{R}^{n\times n}, b \in \mathbb{R}^n, C,D \in \mathbb{R}^{m\times n}, 
d \in \mathbb{R}^m$.

Let $A$ and $DA^{-1}D^T$ be nonsingular. This holds, for example, 
when $A$ is positive definite and $D$ has full rank. 
If $\bar{x}$ is a solution of the QVI, it satisfies the following KKT system:
\begin{equation}\label{sys1}
\begin{cases}
A\bar{x}+b+D^T\bar{\lambda}=0,\\
(C+D)\bar{x}=d,
\end{cases}
\end{equation}
where $\bar{\lambda}\in \mathbb{R}^m$ is a Lagrange multiplier associated 
to the constraints. 

Let us examine the parameterized VI \eqref{paraVI} given by  
\[F(x)=Ax+b+p,\quad K(q)=\{z\in \mathbb{R}^n: Cq+Dz=d\}\]
and Assumption~\ref{m1}.
Solution $x$ of this VI and an associated Lagrange multiplier $\lambda$ satisfy the system
\begin{equation}\label{sys2}
\begin{cases} Ax+b+p+D^T\lambda=0,\\ Cq+Dx=d.  \end{cases} \iff \begin{bmatrix}A&D^T\\D& 0\end{bmatrix}\begin{bmatrix}x\\\lambda\end{bmatrix}=\begin{bmatrix}-b-p\\d-Cq\end{bmatrix}.
\end{equation}
In this system, the Schur complement of block $A$ is $-DA^{-1}D^T.$ The invertibility assumption ensures the solution is unique, meaning that $\mathbf{\MVI}$ is single-valued in the entire domain $\mathbb{R}^n\times\mathbb{R}^n$. Using systems~\eqref{sys1} and~\eqref{sys2} together, we have
\begin{align}
A(x-\bar{x})+D^T(\lambda-\bar{\lambda})=-p,\label{eq1sys3}\\
D(x-\bar{x})=-C(q-\bar{x}).\label{eq2sys3}
\end{align}
Then,
\begin{equation}\label{deltax}
x-\bar{x}=A^{-1}(-p-D^T(\lambda-\bar{\lambda}))=-A^{-1}p-A^{-1}D^T(\lambda-\bar{\lambda}).
\end{equation}
Using this in \eqref{eq2sys3}, we obtain that
\[DA^{-1}p+DA^{-1}D^T(\lambda-\bar{\lambda}) =C(q-\bar{x}),\]
and hence,
\[\lambda-\bar{\lambda}=(DA^{-1}D^T)^{-1}C(q-\bar{x})-(DA^{-1}D^T)^{-1}DA^{-1}p.\]         
Substituting this back into \eqref{deltax}, we conclude that
\begin{eqnarray*}
x-\bar{x} &=& -A^{-1}p-A^{-1}D^T(DA^{-1}D^T)^{-1}C(q-\bar{x})+
              A^{-1}D^T(DA^{-1}D^T)^{-1}DA^{-1}p \\
&=&(-A^{-1}+A^{-1}D^T(DA^{-1}D^T)^{-1}DA^{-1})p-A^{-1}D^T(DA^{-1}D^T)^{-1}C(q-\bar{x}).
\end{eqnarray*}
This implies that
\[\|x-\bar{x}\|\leq L_p\|p\|+L_q\|q-\bar{x}\|,\]
where
\[L_p=\|-A^{-1}+A^{-1}D^T(DA^{-1}D^T)^{-1}DA^{-1}\|,\quad L_q=\|A^{-1}D^T(DA^{-1}D^T)^{-1}C\|.\]
Therefore, Assumption~\ref{m1} is satisfied globally ($V_p=V_q=\mathbb{R}^n$), if 
\[\|A^{-1}D^T(DA^{-1}D^T)^{-1}C\|<\frac{1}{2}.\qed\]
\end{example}

The next examples illustrate the independence between Assumption~\ref{m1} and Assumption~\ref{m2}.

\begin{example}\label{m1notm2}\em
In the QVI~\eqref{QVI2}, let 
$F:\mathbb{R}\to\mathbb{R}$ be defined by $F(x)=x$, and let
$K:\mathbb{R}\rightrightarrows \mathbb{R}$ be defined by
$K(x)=[x^2,+\infty).$

Clearly, $0\in K(0), 1 \in K(1),$ and $N_{K(0)}(0)=N_{K(1)}(1)=(-\infty,0],$ so that
$$0 \in 0+(-\infty,0] = F(0)+N_{K(0)}(0), \text{ and } 0 \in 1+(-\infty,0]=F(1)+N_{K(1)}(1).$$
Thus $\bar{x}_1=0$ and $\bar{x}_2=1$ are solutions of the QVI.
Since the QVI solution is not unique, Assumption~\ref{m2} cannot hold.

We now examine Assumption~\ref{m1} at the point $\bar{x}_1=0$. We have
$$x \in \MVI(p,q) \iff 0 \in p+x+N_{[q^2,+\infty)}(x).$$
This parametric VI has a unique solution because the mapping $x \mapsto p+x$ is strongly monotone.

If $-p>q^2,$ then $x=-p$ is the solution, because 
$$0 \in p-p+N_{[q^2,\infty)}(-p) \iff 0 \in 0+\{0\}.$$

If $-p<q^2,$ then $x=q^2$ is the solution, because
$$0 \in p+q^2+N_{[q^2,+\infty)}(q^2) \iff 0 \in p+q^2+(-\infty,0].$$
Therefore, $$x=\max\{-p,q^2\},\text{ and hence }|x|\leq |p|+|q|^2.$$
Note that
$$|q|^2\leq\frac{1}{3}|q| \iff 3|q|^2-|q|\leq0 \iff |q| \in \Big[0,\frac{1}{3}\Big] \iff q\in \Big[-\frac{1}{3},\frac{1}{3}\Big].$$
Thus, for $(p,q) \in \mathbb{R}\times[-\frac{1}{3},\frac{1}{3}]$,
$$|\MVI(p,q)-\MVI(0,0)|=|x|\leq |p|+|q|^2 \leq |p|+\frac{1}{3}|q|,$$
and Assumption~\ref{m1} is satisfied with $L_p=1,\text{and }L_q=\frac{1}{3}<\frac{1}{2}$.
\end{example}

\begin{example}\label{m2notm1}\em
In the QVI~\eqref{QVI2}, let
$K:\mathbb{R}^n\rightrightarrows\mathbb{R}^n$ be defined by $K(x)=\{-x\},$ and let 
$F:\mathbb{R}^n\to\mathbb{R}^n$ be defined by $F(x)=0$.

Any solution $\bar{x}$ of the QVI satisfies
$$0 \in 0 +N_{\{-\bar{x}\}}(\bar{x}) \implies \bar{x}=-\bar{x} \implies \bar{x}=0.$$
For the parametric QVI,
\begin{eqnarray*}
0 \in \MQVI(p,q) \implies 0 \in p+N_{K(x+q)}(x) &\implies& x = -(x+q) \implies x=-\frac{q}{2},\notag \\
& \implies& L_p=0, L_q=\frac{1}{2}<1.\notag
\end{eqnarray*}
Thus, Assumption~\ref{m2} holds. We next show that Assumption~\ref{m1} does not.

For the parametric VI,
\begin{eqnarray}
0 \in \MVI(p,q) \implies 0 \in p+N_{K(q)}(x) &\implies& x=-q,\notag\\ 
&\implies& L_p=0, L_q=1>\frac{1}{2}.\notag
\end{eqnarray}
Recall that $L_q<\frac{1}{2}$ is required in Assumption~\ref{m1}.
\end{example}

\subsection{Calmness ensures semimonotonicity}

We now show that $\MVI$ or $\MQVI$-calmness ensure the lifted mapping $\mathbf T$ is semimonotone.
\begin{theorem}[Semimonotonicity of QVI mapping]\label{A1AIII}
Under Assumption~\ref{m1} or Assumption~\ref{m2}, there exist $\mu^\perp,\mu>0,$ with $\mu^\perp\mu<1$, and $\mathbf{U} \subset \mathbb{R}^{2n} \times \mathbb{R}^{2n}$ such that the mapping $\mathbf{T}$ in~\eqref{decoupling} is $(\mu^\perp,\mu)$-semimonotone at $(\mathbf{\bar{x}},\mathbf{0})$ on $\mathbf{U}$.
\end{theorem}

\begin{proof}
First, suppose  Assumption~\ref{m1} holds. 
Take $x_1, x_2 \in V_q$ and $y_1 \in V_p$ with $y_1 \in F(x_1)+N_{K(x_2)}(x_1)$, meaning that $x_1=\mathbf{\MVI}(-y_1,x_2)$. Then
\begin{equation}\label{desix1}
\|x_1-\bar{x}\|=
\|\mathbf{\MVI}(-y_1,x_2)-\mathbf{\MVI}(0,\bar{x})\|\leq L_p\|y_1\|+L_q\|x_2-\bar{x}\|.\end{equation}
Therefore,
$$\|x_2-\bar{x}\| \leq \|x_2-x_1\|+\|x_1-\bar{x}\| \leq \|x_2-x_1\|+L_p\|y_1\|+L_q\|x_2-\bar{x}\|,$$
implying that 
\begin{equation}\label{desix2}
\|x_2-\bar{x}\| \leq \frac{L_p}{1-L_q}\|y_1\|+\frac{1}{1-L_q}\|x_2-x_1\|.
\end{equation}
Using the Cauchy-Schwarz inequality and~\eqref{desix1}, we have that 
$$\langle y_1, x_1-\bar{x}\rangle \geq -\|y_1\|\|x_1-\bar{x}\|\geq -\|y_1\|(L_p\|y_1\|+L_q\|x_2-\bar{x}\|)=-L_p\|y_1\|^2-L_q\|y_1\|\|x_2-\bar{x}\|.$$
Using \eqref{desix2}, we conclude that 
\begin{eqnarray}
\langle y_1,x_1-\bar{x}\rangle &\geq& 
-L_p\|y_1\|^2-L_q\|y_1\|\left(\frac{L_p}{1-L_q}\|y_1\| +\frac{1}{1-L_q}\|x_2-x_1\|\right),\notag \\
&=& -\left(L_p+\frac{L_pL_q}{1-L_q}\right)\|y_1\|^2-\frac{L_q}{1-L_q}\|y_1\|\|x_2-x_1\|,\notag \\ 
&=& -\frac{L_p}{1-L_q}\|y_1\|^2-\frac{L_q}{1-L_q}\|y_1\|\|x_1-x_2\|. \label{ineqdot}
\end{eqnarray}
Having in mind the property \eqref{ineqQ},
denoting 
\[Q(x_1,x_2,y_1)=\langle y_1,x_1-\bar{x}\rangle+
\frac{\mu^\perp}{2}\|x_1-x_2\|^2+\frac{\mu}{2}\|y_1\|^2,\]
we need to show $Q(x_1,x_2,y_1)\geq 0.$ 
Using~\eqref{ineqdot}, we obtain that
$$Q(x_1,x_2,y_1) \geq \left(\frac{\mu}{2}-\frac{L_p}{1-L_q}\right)\|y_1\|^2 -
\frac{L_q}{1-L_q}\|y_1\|\|x_1-x_2\|+\frac{\mu^\perp}{2}\|x_1-x_2\|^2.$$
Therefore, $Q(x_1,x_2,y_1)\geq 0$ would hold if the following matrix 
is positive semidefinite: 
$$\mathcal{H}_1=\begin{bmatrix}\frac{\mu}{2}-\frac{L_p}{1-L_q} & \frac{-L_q}{2(1-L_q)} \\ 
\frac{-L_q}{2(1-L_q)} & \frac{\mu^\perp}{2} \end{bmatrix} \succeq 0 .$$ 
Note that 
$$\mathcal{H}_1\succeq 0 \iff \begin{cases} \mu>\frac{2L_p}{1-L_q}, \\ \mu^\perp > 0, \\ \det \mathcal{H}_1=\frac{\mu\mu^\perp}{4}-\frac{\mu^\perp L_p}{2(1-L_q)}-\frac{L_q^2}{4(1-L_q)^2}\geq 0.\end{cases}$$
Since $0<L_q< 1/2,$ we have $\frac{L_q^2}{(1-L_q)^2} <1.$ 
Taking $0<\varepsilon<1-\frac{L_q^2}{(1-L_q)^2}$, we can choose 
\begin{equation}\label{murho}
0<\mu^\perp<\frac{1-L_q}{2L_p}\left(1-\varepsilon-\frac{L_q^2}{(1-L_q)^2}\right)
\text{ and }\mu=\frac{1-\varepsilon}{\mu^\perp}>0.
\end{equation}
For these parameters, 
$$\mu=(1-\varepsilon)\frac{1}{\mu^\perp}>(1-\varepsilon)\frac{2L_p}{(1-L_q)(1-\varepsilon-\frac{L_q^2}{(1-L_q)^2})}>\frac{2L_p}{1-L_q}.$$

We have that $\mu^\perp\mu=1-\varepsilon < 1$. Finally,
\begin{eqnarray*}
\det \mathcal{H}_1 &=& \frac{1-\varepsilon}{4}-\mu^\perp\frac{L_p}{2(1-L_q)}
-\frac{L_q^2}{4(1-L_q)^2} \\
&>& \frac{1-\varepsilon}{4}-\frac{1-L_q}{2L_p}
\left(1-\varepsilon-\frac{L_q^2}{(1-L_q)^2}\right)
\frac{L_p}{2(1-L_q)}-\frac{L_q^2}{4(1-L_q)^2}=0.
\end{eqnarray*}

Since $Q(x_1,x_2,y_1) \geq 0$ when $x_1,x_2 \in V_q$ and $y_1 \in V_p,$ we conclude that $\mathbf{T}$ is $(\mu^\perp,\mu)$-semimonotone at $(\mathbf{\bar{x}},\mathbf{0})$ on $\mathbf{U}=(V_q\times V_q)\times (V_p \times \mathbb{R}^n)$.

Now suppose Assumption~\ref{m2} holds.
 We can choose $\delta>0$ such that $\mathbb{B}_\delta \subset V_p.$ Taking
 $y_1 \in V_p$ and $x_1, x_2 \in \{\bar{x}\}+\frac{1}{2}\mathbb{B}_\delta,$ we have that $\|x_1-x_2\|\leq \|x_1-\bar{x}\|+\|\bar{x}-x_2\|\leq \delta,$ implying that $x_1-x_2 \in V_p$. Therefore,
$$\|x_1-\bar{x}\| = \|\mathbf{\MQVI}(-y_1,x_2-x_1)-\mathbf{\MQVI}(0,0)\|\leq 
L_p\|y_1\|+L_q\|x_2-x_1\|.$$
We obtain that
$$\langle y_1,x_1-\bar{x}\rangle \geq -L_p\|y_1\|^2-L_q\|y_1\|\|x_1-x_2\|,$$
$$Q(x_1,x_2,y_1)\geq \left(\frac{\mu}{2}-L_p\right)\|y_1\|^2-L_q\|y_1\|\|x_1-x_2\|+
\frac{\mu^\perp}{2}\|x_1-x_2\|^2.$$
Therefore, $Q(x_1,x_2,y_1)\geq 0$ if
$$\mathcal{H}_2=\begin{bmatrix}\frac{\mu}{2}-L_p & -\frac{L_q}{2} \\ 
-\frac{L_q}{2} & \frac{\mu^\perp}{2}\end{bmatrix} \succeq 0 .$$
Note that
$$\mathcal{H}_2\succeq 0 \iff \begin{cases} \mu > 2L_p, \\ \mu^\perp> 0, \\ \det\mathcal{H}_2=\frac{\mu\mu^\perp}{4}-\frac{\mu^\perp L_p}{2}-\frac{L_q^2}{4} \geq 0.\end{cases}$$
Since $0\leq L_q<1$ we can take $\varepsilon<1-L_q^2,$ $0<\mu^\perp<\frac{1-L_q^2-\varepsilon}{2L_p}, \mu=\frac{1-\varepsilon}{\mu^\perp},$ so that
$$\mu^\perp\mu=1-\varepsilon < 1,$$
$$\mu=(1-\varepsilon)\frac{1}{\mu^\perp}>(1-\varepsilon)\frac{2L_p}{1-L_q^2-\varepsilon}\geq 2L_p.$$
$$\det\mathcal{H}_2=\frac{1-\varepsilon}{4}-\frac{\mu^\perp L_p}{2}-\frac{L_q^2}{4}>\frac{1-\varepsilon}{4}-\frac{1-L_q^2-\varepsilon}{2L_p}\frac{L_p}{2}-\frac{L_q^2}{4}=0.$$
We conclude that $Q(x_1,x_2,y_1)\geq 0$ when $x_1,x_2 \in \{\bar{x}\}+\frac{1}{2}\mathbb{B}_\delta$ and $y_1 \in V_p$. 
Thus, $\mathbf{T}$ is $(\mu^\perp,\mu)$-semimonotone at $(\mathbf{\bar{x}},\mathbf{0})$ on $\mathbf{U}=((\{\bar{x}\}+\frac{1}{2}\mathbb{B}_\delta)\times (\{\bar{x}\}+\frac{1}{2}\mathbb{B}_\delta))\times (V_p \times \mathbb{R}^n)$.
\end{proof}

\section{Global Convergence Analysis}\label{sec:global}

When  Assumption~\ref{m1} or Assumption~\ref{m2} hold
on the whole domain, we show in Theorem~\ref{globcon} below
that 
Algorithm~\ref{algo2} converges globally. 
Independently, under appropriate conditions, global convergence with linear rate is 
proven for the case of QVIs with moving set, in Theorem~\ref{th:moving}.

\subsection{Global convergence under global calmness}
Before stating our result, we revisit Example~\ref{exeli}, 
to show that the mapping $\mathbf T$ therein satisfies
the semimonotone inequality in Definition~\ref{semimon} in the whole space
(even though neither monotonicity, even local, nor elicitability hold for $\mathbf T$).

\begin{example}[Continuation of Example~\ref{exeli}] \label{exglo}\em
For the same mappings $F$ and $K$
in Example~\ref{exeli}, we shall  
show that $\mathbf{T}$ is semimonotone  at $(\mathbf{\bar{x}},\mathbf{0})=(\mathbf{0},\mathbf{0})$ in two different ways. 

 First, we show semimonotonicity directly using the concept in~\eqref{ineqQ}.
We have that
\[y_1 \in F(x_1)+N_{K(x_2)}(x_1) \implies x_1=t x_2, \text{ for some } 
t\in \mathbb{R},\text{ and } y_1 = x_1+v, \text{ with }v\in[x_2]^\perp.\] 
Then $y_1=t x_2+v, \text{with } v \in [x_2]^\perp,$ and
\begin{align}
\langle y_1, x_1-0\rangle +\frac{\mu^\perp}{2}\|x_1-x_2\|^2+\frac{\mu}{2}\|y_1\|^2=
\langle t x_2+v,t x_2\rangle +\frac{\mu^\perp}{2}\|t x_2-x_2\|^2+
\frac{\mu}{2}\|t x_2+v\|^2 \notag \\
=t^2\|x_2\|^2+\frac{\mu^\perp(t-1)^2}{2}\|x_2\|^2+\frac{\mu t^2}{2}\|x_2\|^2+
\frac{\mu}{2}\|v\|^2\geq 0,\notag
\end{align}
for any choice of $\mu\geq0, \mu^\perp\geq0.$

Alternatively, we can verify that Assumption~\ref{m1} holds,
 and apply Theorem~\ref{A1AIII}. Let $x\in \MVI(p,q)$ be a solution for 
the parametrized VI in question. Then
\[ 0 \in p+F(x)+N_{K(q)}(x) \iff x=t q, \text{for some } t \in \mathbb{R}, \text{and } 
0=p+x+v, \text{with }v \in [q]^\perp.\] 
If $q=0$, we would have $x=0,$ and then $\|\mathbf{\MVI}(p,0)-\mathbf{\MVI}(0,0)\|=0\leq L_q\|q\|+L_p\|p\|,$ for any $L_p,L_q\geq 0.$ So assume $q\neq 0$.
Since $-p=t q+v$, taking the inner product with $q$ we obtain that
\[-\langle p,q\rangle = t\|q\|^2 \implies t
=-\frac{\langle p,q\rangle}{\|q\|^2}\implies x=-\frac{\langle p,q\rangle}{\|q\|^2}q.\]
Therefore, the solution $x$ is unique and
\[ \|\mathbf{\MVI}(p,q)-\mathbf{\MVI}(0,0)\|=\left\|-\frac{\langle p,q\rangle}{\|q\|^2}q\right\|=
\frac{|\langle p,q\rangle|}{\|q\|} \leq \frac{\|p\|\,\|q\|}{\|q\|}= \|p\|\,.\]
Writing the right hand side as $\|p\| = L_q\|q\|+L_p\|p\|$, Assumption~\ref{m1}
holds with $L_q=0<1/2$, $L_p=1$.\qed
\end{example}

We now prove that our QVI progressive decoupling+ method is globally convergent.

\begin{theorem}[Global convergence of Algorithm~\ref{algo2}]\label{globcon}
Assume that $\dom K = \mathbb{R}^n,$ $F$ is continuous and either $K$ is compact-valued
	or $F$ is monotone. Consider parameters in Algorithm~\ref{algo2} satisfying
\begin{equation}\label{parameters}
r\in(\mu^\perp,1/\mu), \lambda_x \in (0,2(1-r\mu)), \lambda_y \in (0,2(1-\mu^\perp/r)),
\end{equation}
where $\mu,\mu^\perp>0$ with $\mu\mu^\perp<1$ are given by Theorem~\ref{A1AIII}. Then the following statements hold:
\begin{enumerate}[label=(\roman*)]
\item
If Assumption~\ref{m1} holds with $V_p=V_q=\mathbb{R}^n$, then either a solution of QVI~\eqref{QVI} 
		is reached in a finite number of iterations, or the generated sequence
		$\{x^k\}$ is bounded and 
	every accumulation point of this sequence
		is a solution of QVI~\eqref{QVI}. 
		In particular, 
		if $\bar{x}$ is the unique solution of QVI~\eqref{QVI}, 
		then the sequence converges to $\bar{x}$.
\item
If Assumption~\ref{m2} holds with $V_p=V_q=\mathbb{R}^n$, 
		then the generated sequence $\{x^k\}$
		reaches a solution in a finite number of iterations or 
		converges to the (unique in this case) QVI solution $\bar{x}$.
\end{enumerate}
\end{theorem}

\begin{proof}
Global convergence of progressive decoupling+ depends on three properties stated in 
	Assumption IV of \cite{evens2025}, referred to as A1, A2, and A3. 
	We next show that these hold in our case.

	Property A1 is semimonotonicity, which follows from Theorem~\ref{A1AIII} 
	with $\mathbf{U}=\mathbb{R}^{2n}\times \mathbb{R}^{2n}$,
	under Assumption~\ref{m1} or under Assumption~\ref{m2}.

Property A2 concerns outer semicontinuity of the mapping $\mathbf{T}$, which holds
	here on the entire domain. To see this, let
\[\left(\begin{bmatrix} x_1^i \\ x_2^i\end{bmatrix},
\begin{bmatrix}y_1^i \\ y_2^i \end{bmatrix}\right) \in \graph \mathbf{T}
\text{ be such that }
\left(\begin{bmatrix} x_1^i \\ x_2^i\end{bmatrix},
\begin{bmatrix}y_1^i \\ y_2^i \end{bmatrix}\right)\to
\left(\begin{bmatrix} \bar{x}_1 \\ \bar{x}_2\end{bmatrix}, \begin{bmatrix}\bar{y}_1 \\
\bar{y}_2 \end{bmatrix}\right),
\quad \text{as } i\to\infty.\]
Clearly, we have that
$y_1^i \in F(x_1^i)+N_{K(x_2^i)}(x_1^i)$ and $y_2^i=0.$ Thus $\bar{y}_2=\lim\limits_{i\to \infty} y_2^i=0$. For any $z \in K(\bar{x}_2)$, by the
inner semicontinuity of $K(\cdot)$, we can choose $z^i \in K(x_2^i)$ with $z^i \to z$. Since $y_1^i-F(x_1^i)\in N_{K(x_2^i)}(x_1^i)$, it holds that
$\langle y_1^i-F(x_1^i), z^i-x_1^i\rangle \leq 0.$
By the continuity of $F(\cdot)$, taking the limit as $i\to\infty$, we obtain that
\[\langle \bar{y}_1-F(\bar{x}_1),z-\bar{x}_1\rangle \leq 0.\]
Since $z\in K(\bar{x}_2)$ was arbitrary, it follows that
$\bar{y}_1 \in F(\bar{x}_1)+N_{K(\bar{x}_2)}(\bar{x}_1).$
Therefore,
\[\left(\begin{bmatrix}\bar{x}_1 \\ \bar{x}_2 \end{bmatrix},
\begin{bmatrix} \bar{y}_1 \\ \bar{y}_2\end{bmatrix}\right) \in \graph \mathbf{T}.\]
Hence, $\graph \mathbf{T}$ is closed, and therefore $\mathbf{T}$ is outer semicontinuous.

Finally, property A3 requires the resolvent $\mathcal{J}_{\frac1r\mathbf{T}}$ to have full domain. 
Indeed, for $\mathbf{x} \in \mathbb{R}^{2n}$,
\begin{eqnarray*}
\mathbf{y} \in \mathcal{J}_{\frac1r\mathbf{T}}(\mathbf{x}) &\iff&
\mathbf{x} \in \mathbf{y}+\frac{1}{r}\mathbf{T}(\mathbf{y}) \iff
r(\mathbf{x}-\mathbf{y}) \in \mathbf{T}(\mathbf{y}) \\
&\iff& \begin{cases}
r(x_1-y_1) \in F(y_1)+N_{K(y_2)}(y_1),\\
r(x_2-y_2)=0.
\end{cases}
\end{eqnarray*}
Or equivalently,
$y_2=x_2$ and  $y_1$ solves VI$(F+rI-rx_1,K(x_2))$.
Since $x_2 \in \dom K$, we have that $K(x_2)\neq \emptyset$. 
If $K$ is compact-valued, then $K(x_2)$ is compact, and existence of a solution to $\mathrm{VI}(F+rI-rx_1,K(x_2))$ follows from the
continuity of $F$. If $F$ is monotone, then $F+rI-rx_1$ is strongly monotone, 
and therefore the VI has the unique solution. In either case, 
$\mathcal{J}_{\frac1r\mathbf{T}}(\mathbf{x})\neq \emptyset$ for all $\mathbf{x} \in \mathbb{R}^{2n}$. Thus, $\mathcal{J}_{\frac1r\mathbf{T}}$ has full domain. 

When the parameters lie in the nonempty sets given by \eqref{parameters}, 
Corollary 4.10 of \cite{evens2025} ensures that Algorithm~\ref{algo2} either 
reaches a solution in a finite number of iterations or generates a bounded sequence whose 
accumulation points are solutions of QVI~\ref{QVI} 
(see also item (iii) of Theorem 4.9 in \cite{evens2025}). 
If the solution is unique, then the sequence converges to this solution. 
The case of Assumption~\ref{m2} falls into the uniqueness setting.
\end{proof}

Convergence of Algorithm~\ref{algo1} follows  
taking $\lambda_x=\lambda_y=1$ in \eqref{parameters}.

\begin{corollary}[Global convergence of Algorithm~\ref{algo1} with fixed proximal parameter]\label{globcon-alg1}
\mbox{}

Consider the QVI progressive decoupling method given in Algorithm~\ref{algo1}.
Let for all iterations,  $r_k=r\in (\mu^\perp,1/\mu)$. 
Then, if the mappings $F$ and $K$ are as in Theorem~\ref{globcon}, 
under the assumptions in $(i)$ and $(ii)$ therein, the convergence properties of
Algorithm~\ref{algo1} with a fixed proximal parameter 
are the same as stated in Theorem~\ref{globcon} for Algorithm~\ref{algo2}.\qed
\end{corollary}

\subsection{The moving set case}

 For the class of QVIs whose mapping $K$ represents a moving set: $K(x)=c(x)+C$,
where $C\subset \mathbb{R}^n$ is a fixed closed convex set and
$c:\mathbb{R}^n\to\mathbb{R}^n$ is a given function,
we next establish convergence of Algorithm~\ref{algo1} directly, allowing
variable proximal parameters, without invoking
the framework of \cite{evens2025}.         
Moreover, under suitable conditions our algorithm converges with 
the global linear rate.

Formulations of QVI with a moving set appear in numerous mechanical 
applications with implicit-obstacle-type constraints, as described in \cite{Kravchuk_2007}.  This specific structure
has been extensively studied in the QVI literature; see \cite{Facchinei_2013}. A proximal point method for QVIs with moving set is proposed in \cite{Mijajlovi2015}. 

We state the following proposition, which will be used in the sequel. 
It is an immediate consequence of some results in \cite{Nesterov_2006}.
Specifically, \cite[Lemma~2]{Nesterov_2006} on a certain property of projections
onto a moving set, and \cite[Corollary~2]{Nesterov_2006} 
on existence of solutions for a general QVI
which requires a condition on projections (satisfied under appropriate 
assumptions in the moving set case).

\begin{proposition}\label{propnest}
For the QVI \eqref{QVI}, let $F$ be strongly monotone with modulus $\sigma >0$ and 
Lipschitz continuous with  modulus $\ell_F>0$. Let              
$K:\mathbb{R}^n\rightrightarrows \mathbb{R}^n$ be given by $K(x) = c(x)+C,$ where 
$c:\mathbb{R}^n \to \mathbb{R}^n$ is Lipschitz continuous with modulus $\ell_c>0,$ and 
$C\subset \mathbb{R}^n$ is a closed, convex set.  
If $\ell_c < \frac{\sigma}{\ell_F}$, then the QVI is solvable, and the solution is unique. 
\end{proposition}

We now introduce an assumption required to establish global convergence of Algorithm~\ref{algo1} for the moving set case.
\begin{Assum}[Parameters relations for a QVI with moving set]\label{contract}
Let QVI \eqref{QVI} be such that:
\begin{enumerate}
\item $F$ is $\sigma$-strongly monotone and Lipschitz continuous with  modulus $\ell_F>0$;
\item $K(x)=c(x)+C$,  where $C\subset \mathbb{R}^n$ is a closed convex set 
and $c:\mathbb{R}^n\to \mathbb{R}^n$ is Lipschitz continuous with modulus $\ell_c>0$;
\item For some $\tau\in (\frac{1}{2},1)$,
$$2\sigma-\ell_F^2>0\; \text{ and }\;
\ell_c\leq(\sqrt{\tau}-\frac{\sqrt{2}}{2})\sqrt{2\sigma-\ell_F^2}.\qed$$
\end{enumerate}
\end{Assum}

Before stating the next theorem, note that $\sqrt{2\tau}-\tau \in (0,\frac{1}{2})$ because 
\[\begin{cases}\sqrt{2\tau}-\tau = \sqrt{2}\sqrt{\tau}-\tau \geq \sqrt{2}\tau-\tau > 0,\\
\sqrt{2\tau}-\tau = \sqrt{2}\sqrt{\tau}-\tau \leq \max\limits_{u \in (\frac{\sqrt{2}}{2},1)}\{\sqrt{2}u-u^2\}<\sqrt{2}\left(\frac{\sqrt{2}}{2}\right)-\left(\frac{\sqrt{2}}{2}\right)^2=\frac{1}{2}.
\end{cases}\]
Now, we can state the following result.

\begin{theorem}\label{th:moving}
[Convergence of Algorithm~\ref{algo1} with varying proximal parameter for moving set QVIs] 
Under Assumption~\ref{contract}, the QVI \eqref{QVI} has the unique solution $\bar{x}$, and
	from any starting point, Algorithm~\ref{algo1} generates a sequence $\{x^k\}$ converging 
to this solution, provided the parameters $r_k$ satisfy
$0<r_0\leq r_k\leq r_{k+1} \leq \left(\frac{\frac{1}{2}-(\sqrt{2\tau}-\tau)}{\frac{1}{2}+(\sqrt{2\tau}-\tau)}\right)(2\sigma-\ell_F^2)$, for all $k.$ 

Moreover, the primal-dual sequence
$\{\|x^k-\bar{x}\|^2+\frac{1}{r_k^2}\|y^k\|^2\}$ converges to zero
with a global linear rate.
\end{theorem}
\begin{proof}
First, we establish the existence and uniqueness of the QVI solution. Note that $\frac{\sigma}{\ell_F}>0$, and 
	\begin{eqnarray*}\frac{\sqrt{2\sigma-\ell_F^2}}{\frac{\sigma}{\ell_F}} &=&
	\sqrt{2\ell_F^2\frac{1}{\sigma}-\ell_F^4\frac{1}{\sigma^2}}\\
	&\leq& \max\limits_{\frac{1}{\sigma}>0}\sqrt{2\ell_F^2\left(\frac{1}{\sigma}\right)-
	\ell_F^4\left(\frac{1}{\sigma}\right)^2} \\
	&=&\sqrt{2\ell_F^2\left(\frac{1}{\ell_F^2}\right)-\ell_F^4\left(\frac{1}{\ell_F^2}\right)^2}=1.
\end{eqnarray*}
By the inequality above and Assumption~\ref{contract}, it follows that $\ell_c\leq (\sqrt{\tau}-\frac{\sqrt{2}}{2})\frac{\sigma}{\ell_F}<\frac{\sigma}{\ell_F}.$ Proposition~\ref{propnest} ensures the existence and uniqueness of the QVI solution.

Let $\bar{x}$ be the solution of the QVI.
By the definition of the iterates of Algorithm \ref{algo1}
and since $0 < r_k\leq r_{k+1}$, we obtain that 
\begin{eqnarray*}
\|x^{k+1}-\bar{x}\|^2+\frac{1}{r_{k+1}^2}\|y^{k+1}\|^2
&\leq& \|x^{k+1}-\bar{x}\|^2+\frac{1}{r_k^2}\|-\frac{r_k}{2}(\hat{x}^k-x^k-\frac{1}{r_k}y^k)\|^2\\
&=&           
\|x^{k+1}-\bar{x}\|^2+\|\frac{1}{2}(\hat{x}^k+x^k-\frac{1}{r_k}y^k)-x^k\|^2 \\
&=& \|x^{k+1}-\bar{x}\|^2+\|x^{k+1}-x^k\|^2 \\       
&=& \frac{1}{2}\left(\|(x^{k+1}-\bar{x})+(x^{k+1}-x^{k})\|^2+
\|(x^{k+1}-\bar{x})-(x^{k+1}-x^k)\|^2\right) \\
&=&           
\frac{1}{2}\left(\|2x^{k+1}-x^k-\bar{x}
\|^2+\|x^k-\bar{x}\|^2\right) \\
&=& \frac{1}{2}\left(\|\hat{x}^k-\frac{1}{r_k}y^k-\bar{x}\|^2+\|x^k-\bar{x}\|^2\right) \\
&=&           
\frac{1}{2}(\|\hat{x}^k-\bar{x}\|^2+\|x^k-\bar{x}\|^2+
\frac{1}{r_k^2}\|y^k\|^2-2\langle \hat{x}^k-\bar{x}, \frac{1}{r_k}y^k\rangle) \\
&=& \frac{1}{2}(\|x^k-\bar{x}\|^2+\frac{1}{r_k^2}\|y^k\|^2)+
\frac{1}{2}\|\hat{x}^k-\bar{x}\|^2-\langle \hat{x}^k-\bar{x},\frac{1}{r_k}y^k\rangle.      
\end{eqnarray*}
We next analyze the last term in the right-hand side of the relation above.
Note that for any $\theta_k\neq 0$ (this $\theta_k$ is introduced here
for a reason that will become clear in the sequel), it holds that
$$-\langle \hat{x}^k-\bar{x},\frac{1}{r_k}y^k\rangle = 
-\langle \theta_k(\hat{x}^k-\bar{x}),\frac{1}{\theta_kr_k}y^k\rangle \leq 
\frac{\theta_k^2}{2}\|\hat{x}^k-\bar{x}\|^2+\frac{1}{2\theta_k^2r_k^2}\|y^k\|^2.$$ 
Combining the relations above, we obtain that
\begin{equation}
\|x^{k+1}-\bar{x}\|^2+\frac{1}{r_{k+1}^2}\|y^{k+1}\|^2 \leq 
\frac{1}{2}\|x^k-\bar{x}\|^2+\frac{1+\theta_k^2}{2}\|\hat{x}^k-\bar{x}\|^2+\left(\frac{1}{2}+\frac{1}{2\theta_k^2}\right)\frac{1}{r_k^2}\|y^k\|^2. \label{des_} \end{equation}
	The task now is to estimate the middle term in the right-hand side of \eqref{des_},
	so that the expression reduces to a suitable multiple of 
	$\|x^k-\bar{x}\|^2+\frac{1}{r_k^2}\|y^k\|^2$, implying (linear) convergence
	of the corresponding sequence.

	Observe that
\[\bar{x} \in K(\bar{x})=c(\bar{x})+C \implies \bar{x}-c(\bar{x})+c(x^k-\frac{1}{r_k}y^k) \in 
K(x^k-\frac{1}{r_k}y^k) = K^k.\]
Therefore,
the fact that $\hat{x}^k$ solves
VI$(F+ r_k I-r_kx^k-y^k, K^k)$ implies that
\begin{equation}\label{aux1}
\langle F(\hat{x}^k)+r_k\hat{x}^k-r_kx^k-y^k, \bar{x}-
c(\bar{x})+c(x^k-\frac{1}{r_k}y^k) - \hat{x}^k\rangle \geq 0 .
\end{equation}
On the other hand,
\[\hat{x}^k\in K(x^k-\frac{1}{r_k}y^k)=c(x^k-\frac{1}{r_k}y^k)+C \implies 
\hat{x}^k-c(x^k-\frac{1}{r_k}y^k)+c(\bar{x}) \in K(\bar{x}).\]
The fact that
$\bar{x}$ solves 
QVI$(F,K(\cdot))$ implies that
\[\langle F(\bar{x}), \hat{x}^k-c(x^k-\frac{1}{r_k}y^k)+c(\bar{x})-\bar{x}\rangle \geq 0,\]
or equivalently,
\[\langle -F(\bar{x}), \bar{x}-
c(\bar{x})+c(x^k-\frac{1}{r_k}y^k) - \hat{x}^k\rangle \geq 0.\]
Adding the last inequality above with \eqref{aux1}, we obtain that
\[\langle F(\hat{x}^k)-F(\bar{x})+r_k\hat{x}^k-r_kx^k-y^k,\bar{x}-\hat{x}^k+c(x^k-\frac{1}{r_k}y^k)-c(\bar{x})\rangle \geq 0.\]
Hence, by the strong monotonicity of $F$, we obtain that
\begin{eqnarray}\label{desigrande}
\sigma\|\hat{x}^k-\bar{x}\|^2 &\le& \langle F(\hat{x}^k)-F(\bar{x}), \hat{x}^k - \bar{x}\rangle 
\notag \\
&\le& \langle F(\hat{x}^k)-F(\bar{x}), c(x^k-\frac{1}{r_k}y^k)-c(\bar{x})\rangle + 
r_k\langle \hat{x}^k-x^k-\frac{1}{r_k}y^k,\bar{x}-\hat{x}^k \rangle \notag \\ 
&& +r_k\langle \hat{x}^k-x^k-\frac{1}{r_k}y^k, c(x^k-\frac{1}{r_k}y^k)-c(\bar{x})\rangle .
\end{eqnarray}
We next estimate the three terms in the right-hand side of \eqref{desigrande}.

 First, by the Lipschitz continuity of $F$ and of $c$, we have that
\begin{eqnarray}\label{3ineq-1}
\langle F(\hat{x}^k)-F(\bar{x}), c(x^k-\frac{1}{r_k}y^k)-c(\bar{x})\rangle 
&\leq& \frac{1}{2} \|F(\hat{x}^k)-F(\bar{x})\|^2 +
\frac{1}{2} \|c(x^k-\frac{1}{r_k}y^k)-c(\bar{x})\|^2 \notag \\ 
&\le & \frac{\ell_F^2}{2}\|\hat{x}^k-\bar{x}\|^2+\frac{\ell_c^2}{2}\|x^k-\frac{1}{r_k}y^k-\bar{x}\|^2.
\end{eqnarray}

Second, it holds that                                     
\begin{equation}\label{3ineq-2}
r_k\langle\hat{x}^k-x^k-\frac{1}{r_k}y^k,\bar{x}-\hat{x}^k\rangle= 
\frac{r_k}{2}(\|\bar{x}-x^k-\frac{1}{r_k}y^k\|^2-\|\hat{x}^k-x^k-
\frac{1}{r_k}y^k\|^2-\|\hat{x}^k-\bar{x}\|^2).
\end{equation}

Third, using again the Lipschitz continuity of $c$, we obtain that
\begin{equation}\label{3ineq-3}
r_k\langle\hat{x}^k-x^k-\frac{1}{r_k}y^k, c(x^k-\frac{1}{r_k}y^k)-c(\bar{x})\rangle \leq 
\frac{r_k}{2}\|\hat{x}^k-x^k-\frac{1}{r_k}y^k\|^2+
\frac{r_k\ell_c^2}{2}\|x^k-\frac{1}{r_k}y^k-\bar{x}\|^2.
\end{equation}

Combining now \eqref{3ineq-1}--\eqref{3ineq-3} with \eqref{desigrande},
we obtain that
$$ \sigma\|\hat{x}^k-\bar{x}\|^2 \le 
\frac{\ell_F^2-r_k}{2}\|\hat{x}^k-\bar{x}\|^2+\frac{\ell_c^2(r_k+1)}{2}\|x^k-\bar{x}-
\frac{1}{r_k}y^k\|^2+\frac{r_k}{2}\|x^k-\bar{x}+\frac{1}{r_k}y^k\|^2.$$
As $r_k+2\sigma-\ell_F^2>0$, it follows that
$$\|\hat{x}^k-\bar{x}\|^2 \leq \frac{\ell_c^2(r_k+1)}{r_k+2\sigma-\ell_F^2}\|x^k-\bar{x}-\frac{1}{r_k}y^k\|^2+\frac{r_k}{r_k+2\sigma-\ell_F^2}\|x^k-\bar{x}+\frac{1}{r_k}y^k\|^2.$$
Denoting $\gamma_k=\max\{r_k,\ell_c^2(r_k+1)\}$, it then holds that
\begin{eqnarray*}
\|\hat{x}^k-\bar{x}\|^2 &\leq& 
\frac{\gamma_k}{r_k+2\sigma-\ell_F^2}(\|x^k-\bar{x}-\frac{1}{r_k}y^k\|^2+
\|x^k-\bar{x}+\frac{1}{r_k}y^k\|^2) \\
&=&\frac{2\gamma_k}{r_k+2\sigma-\ell_F^2}(\|x^k-\bar{x}\|^2+\frac{1}{r_k^2}\|y^k\|^2).
\end{eqnarray*}
Using this inequality in \eqref{des_}, we obtain that
\begin{eqnarray*}
\|x^{k+1}-\bar{x}\|^2+\frac{1}{r_{k+1}^2}\|y^{k+1}\|^2 
&\leq& 
\left(\frac{1}{2}+\frac{(1+\theta_k^2)\gamma_k}{(r_k+2\sigma-\ell_F^2)}\right)\|x^k-\bar{x}\|^2 \\
&& +
\left(\frac{1}{2}+\frac{1}{2\theta_k^2} 
+ \frac{(1+\theta_k^2)\gamma_k}{(r_k+2\sigma-\ell_F^2)}\right)
\frac{1}{r_k^2}\|y^k\|^2 \\
&\le& \left(\frac{1}{2}+\frac{1}{2\theta_k^2}+
\frac{(1+\theta_k^2)\gamma_k}{(r_k+2\sigma-\ell_F^2)}\right)
	(\|x^k-\bar{x}\|^2+\frac{1}{r_k^2}\|y^k\|^2) .
\end{eqnarray*}          
To establish linear convergence
of the sequence $\{\|x^k-\bar{x}\|^2+\frac{1}{r_k^2}\|y^k\|^2\}$ to zero,
we have to show that the multiple in the right-hand side of the last relation
is bounded above by some $t<1$ for all $k$.

Take $\theta_k=\left(\frac{r_k+2\sigma-\ell_F^2}{2\gamma_k}\right)^\frac{1}{4}$ 
and denote $\beta_k=\left(\frac{\gamma_k}{r_k+2\sigma-\ell_F^2}\right)^\frac{1}{2}$.
Note that $\theta_k^2 =\frac{1}{\sqrt{2}\beta_k}$.
We have that
\begin{eqnarray*}\left(\frac{1}{2}+\frac{1}{2\theta_k^2}+
\frac{(1+\theta_k^2)\gamma_k}{(r_k+2\sigma-\ell_F^2)}\right)
	&=& \frac{1}{2}+\frac{\sqrt{2}}{2} \beta_k+\beta_k^2 +\frac{1}{\sqrt{2}\beta_k} \beta_k^2\\
	&=& \frac{1}{2}+\sqrt{2}\beta_k+\beta_k^2 = \left(\beta_k + \frac{\sqrt{2}}{2} \right)^2. 
\end{eqnarray*}
Hence, the claim would follow if we show, for example, that
\begin{equation}\label{betaine}
\left(\beta_k + \frac{\sqrt{2}}{2} \right)^2 \leq \tau, \text{ for all } k,
\end{equation}
where $\tau \in (1/2,1)$ is given in the assumptions of the Theorem. 
Since $\beta_k>0$, the condition~\eqref{betaine} is equivalent to
\begin{equation}\label{betaine2}
\beta_k\leq \sqrt{\tau}-\frac{\sqrt{2}}{2}.
\end{equation}
To conclude the proof, it remains to verify \eqref{betaine2}.

As $0<\sigma \leq \ell_F \text{ and } 2\sigma-\ell_F^2 \geq 0$, it holds that 
\[2\sigma-\ell_F^2 \leq 2\ell_F-\ell_F^2 \leq \max\limits_{u \in \mathbb{R}}\{2u-u^2\}=1.\]
Since $2\sigma-\ell_F^2\leq 1$, we have that
$r_k(2\sigma-\ell_F^2)+2\sigma-\ell_F^2 \leq r_k+2\sigma-\ell_F^2$. Hence,
\begin{equation}\label{ineqrplus1}
 (r_k+1)(2\sigma-\ell_F^2) \leq r_k+2\sigma-\ell_F^2.
\end{equation}
Squaring the second inequality in Assumption~\ref{contract}, we obtain that
\[\ell_c^2\leq (\tau-\sqrt{2\tau}+\frac{1}{2})(2\sigma-\ell_F^2).\]
Multiplying the latter by $(r_k+1)$, we have
\begin{eqnarray}
	\ell_c^2(r_k+1) &\leq& (\tau-\sqrt{2\tau}+\frac{1}{2})(2\sigma-\ell_F^2)(r_k+1),\notag \\ 
	&\leq& (\tau-\sqrt{2\tau}+\frac{1}{2})(r_k+2\sigma-\ell_F^2),\label{inegamma1}
\end{eqnarray}
where the second inequality is by \eqref{ineqrplus1}.
Using the condition on the choice of $r_k$, which is
\[r_k \leq \left(\frac{\frac{1}{2}-(\sqrt{2\tau}-\tau)}
{\frac{1}{2}+(\sqrt{2\tau}-\tau)}\right)(2\sigma-\ell_F^2),\]
we have that
\[r_k\left(\frac{1}{2}+(\sqrt{2\tau}-\tau)\right) \leq 
\left(\frac{1}{2}-(\sqrt{2\tau}-\tau)\right)(2\sigma-\ell_F^2),\]
from where it follows that
\[  r_k\left(\frac{1}{2}+(\sqrt{2\tau}-\tau)\right)+r_k\left(\frac{1}{2}-(\sqrt{2\tau}-\tau)\right)\leq 
\left(\frac{1}{2}-(\sqrt{2\tau}-\tau)\right)(2\sigma-\ell_F^2+r_k),\]
and hence,      
\begin{equation}
r_k \leq (\tau-\sqrt{2\tau}+\frac{1}{2})(r_k+2\sigma-\ell_F^2).\label{inegamma2}
\end{equation}
Since $\gamma_k=\max\{r_k,\ell_c^2(r_k+1)\},$ inequalities~\eqref{inegamma1} and~\eqref{inegamma2} yield
\begin{align}
\gamma_k\leq(\tau-\sqrt{2\tau}+\frac{1}{2})(r_k+2\sigma-\ell_F^2) &\implies \frac{\gamma_k}{r_k+2\sigma-\ell_F^2}\leq \left(\sqrt{\tau}-\frac{\sqrt{2}}{2}\right)^2 \notag \\ &\implies \beta_k =\left(\frac{\gamma_k}{r_k+2\sigma-\ell_F^2}\right)^{\frac{1}{2}}\leq \sqrt{\tau}-\frac{\sqrt{2}}{2}.\notag
\end{align}
This verifies \eqref{betaine2}, thus completing the proof.
\end{proof}

To conclude with the convergence theory for Algorithm~\ref{algo2},
we now examine its behavior when the
regularity conditions required in Theorem~\ref{globcon} only hold near a solution.
\section{Local Convergence Analysis}\label{sec:loc}

To exhibit local convergence of the progressive decoupling+ algorithm,
a localization condition (the region in which convergence assertions hold), 
must be specified.
To this end, following \cite{evens2025},
we define the following neighborhood of a solution.

\begin{equation}\label{We}
\mathcal{W}_\varepsilon=\left\{(\mathbf{x},\mathbf{y}) \in \mathbf{S} \times \mathbf{S}^\perp : 
	\frac{r}{\lambda_x}\|\mathbf{x}-\mathbf{\bar{x}}\|^2+
	\frac{1}{r\lambda_y}\|\mathbf{y}-\mathbf{0}\|^2\leq \varepsilon^2\right\}.
\end{equation}
Equivalently, 
\[
(\mathbf{x},\mathbf{y}) \in \mathcal{W}_\varepsilon \iff \mathbf{x}=\begin{bmatrix}x\\x \end{bmatrix}, \mathbf{y}=\begin{bmatrix}y\\-y \end{bmatrix}, \text{ with } \frac{2r}{\lambda_x}\|x-\bar{x}\|^2+\frac{2}{r\lambda_y}\|y\|^2\leq \varepsilon^2.
\]
In particular, 
\begin{equation}
(\mathbf{x},\mathbf{y}) \in \mathcal{W}_\varepsilon \implies \mathbf{x}=\begin{bmatrix}x\\x \end{bmatrix}, \mathbf{y}=\begin{bmatrix}y\\-y \end{bmatrix}, \|x-\bar{x}\|\leq \varepsilon\sqrt{\frac{\lambda_x}{2r}} \text{ and }\|y\|\leq \varepsilon\sqrt{\frac{r\lambda_y}{2}}.\label{desiaux}
\end{equation}

In addition to the Assumption~\ref{m1} or Assumption~\ref{m2}, we introduce another 
regularity condition, now on the parameterized VI with a 
proximal term. That it is independent from Assumption~\ref{m1} or Assumption~\ref{m2},
is illustrated in Examples~\ref{calmrnotm1m2} and \ref{m1m2notcalmr}, shortly in the sequel.
This new condition is natural, as the 
proximal term tends to stabilize solutions near the 
reference point $\bar{x}$. 
Given $r>0$ and $\bar{x}$ a solution of the QVI~\ref{QVI2}, we introduce the solution mapping $\mathbf{M^r}:\mathbb{R}^n\times \mathbb{R}^n\rightrightarrows \mathbb{R}^n,$
\begin{equation}\label{paraVIr} 
x \in \mathbf{M^r}(p, q) \iff 0 \in p+F(x)+r(x-\bar{x})+N_{K(q)}(x).
\end{equation}  

\begin{Assum}[$\mathbf{M^r}$-calmness]\label{calmr}
The mapping $\mathbf{M^r}$ is locally single-valued and calm at $((0,\bar{x}),\bar{x})$:
	there exist neighborhoods $V_p' \ni 0, V_q' \ni \bar{x}$ and constants $L_p',L_q' \geq 0$ such that $\mathbf{M^r}$ is single-valued in $V_p'\times V_q'$ and
\[\|\mathbf{M^r}(p,q)-\mathbf{M^r}(0,\bar{x})\| \leq L_p'\|p\|+L_q'\|q-\bar{x}\|, \text{ for all } (p,q) \in V_p'\times V_q'.\]
\end{Assum}

Note that we do not require bounds for the calmness constants in Assumption~\ref{calmr}. 
In this case, single-valuedness follows from strong monotonicity induced by the proximal term, 
as guaranteed by Theorem 2.3.3-(b) from~\cite{2004}. 
Subregularity is more easily satisfied, in particular because there are no bounds required for the 
calmness constants. Results from \cite{Dontchev_2004, Dontchev_2021} can be useful to analyze  
subregularity further.

Assumption~\ref{calmr} is independent of Assumptions~\ref{m1} and~\ref{m2}, as illustrated by the following examples.

\begin{example}\label{calmrnotm1m2}\em
In the QVI~\eqref{QVI2}, let
$K:\mathbb{R}^n\rightrightarrows \mathbb{R}^n$ be the constant mapping $K(x)=\mathbb{R}^n$ and 
let $F:\mathbb{R}^n\to\mathbb{R}^n$ be defined by $F(x)=-x$.

If $\bar{x}$ solves the QVI, then
$$0 \in -\bar{x}+\{0\} \implies \bar{x}=0.$$
Analyzing the parametric QVI,
$$0\in\MQVI(p,q) \implies 0 \in p-x+\{0\} \implies x=p, L_p=1, L_q=0<1.$$
For the parametric VI,
$$0 \in \MVI(p,q) \implies 0 \in p-x+\{0\} \implies x=p, L_p=1, L_q=0<\frac{1}{2}.$$
Hence, Assumptions~\ref{m1} and~\ref{m2} hold.
However, analyzing the mapping in~\eqref{paraVIr} with $r=1$, 
$$x \in \mathbf{M^r}(p,q) \implies 0 \in p-x+x+\{0\} \implies p=0,x \in \mathbb{R}^n.$$
Therefore, 
$$\mathbf{M^r}(p,q)=\begin{cases}\mathbb{R}^n,\text{ if }p=0,\\ \emptyset,\text{ if }p\neq 0. \end{cases}$$
Clearly, Assumption~\ref{calmr} is not satisfied.
\end{example}

\begin{example}\label{m1m2notcalmr}\em
In QVI~\eqref{QVI2}, let
$K:\mathbb{R}^n\rightrightarrows\mathbb{R}^n$ be defined by $K(x)=x+\mathbb{R}_+^n$ and 
$F:\mathbb{R}^n\to\mathbb{R}^n$ be the identity $F(x)=x$.
 
Any point in the nonnegative orthant $x\in\mathbb{R}_+^n$ is a solution of the QVI, because
$$0 \in F(x)+N_{K(x)}(x) \iff 0 \in x+\mathbb{R}_-^n \iff x \in \mathbb{R}_+^n$$
We can choose a solution to analyze the local assumptions. Take
 $\bar{x}=(1,\cdots,1).$ The lack of uniqueness of the QVI solution implies Assumption~\ref{m2}
 cannot be satisfied; we next consider Assumption~\ref{m1}.

Take $(p,q)$ such that $\|p\|\leq \frac{1}{2}, \|q-\bar{x}\|\leq \frac{1}{2}$.
We have that
\begin{equation}\label{paraVIex}
0 \in \MVI(p,q) \iff 0 \in p+x+N_{K(q)}(x).
\end{equation}
It is easy to verify that $x=q$ satisfies~\eqref{paraVIex}. In fact,
\begin{eqnarray}
|p_i|\leq \|p\|\leq \frac{1}{2} &\implies& p_i\geq-\frac{1}{2},\notag\\
|q_i-\bar{x}_i|\leq \|q-\bar{x}\| \leq\frac{1}{2} \implies q_i-\bar{x}_i\geq-\frac{1}{2} &\implies& q_i \geq \frac{1}{2},\notag\\
p_i+q_i \geq 0 \implies p+q \in \mathbb{R}_+^n &\implies& 0 \in p+q+\mathbb{R}_-^n=p+q+N_{K(q)}(q).\notag
\end{eqnarray}
Thus, by the strong monotonicity of $x\mapsto p+x,$ we conclude that $q$ is the unique solution:
 $q=\MVI(p,q).$ 
Although we have local single-valuedness of $\MVI$, Assumption~\ref{m1} is not satisfied,
because $L_q=1>\frac{1}{2}$, contrary to required. 

We now examine the mapping in~\eqref{paraVIr}. We have that
\begin{equation}\label{paraVIrex}
x \in \mathbf{M^r}(p,q) \iff 0 \in x+p+r(x-\bar{x})+N_{K(q)}(x).
\end{equation}
Take $(p,q)$ such that $\|p\|<\frac{1}{4}, \|q-\bar{x}\|<\min\{\frac{1}{4},\frac{1}{4r}\}$.
Taking $x=q$ in~\eqref{paraVIrex}, we have $N_{K(q)}(q)=\mathbb{R}_-^n$. Note that 
\begin{eqnarray}
\|p+r(q-\bar{x})\|\leq \|p\|+r\|q-\bar{x}\|< \frac{1}{2} &\implies& |p_i+r(q_i-\bar{x}_i)|<\frac{1}{2} \implies p_i+r(q_i-\bar{x}_i) > -\frac{1}{2},\notag\\
\|q-\bar{x}\|<\frac{1}{4} \implies |q_i-\bar{x}_i|<\frac{1}{4} &\implies& q_i-\bar{x}_i > -\frac{1}{4} \implies q_i > 1-\frac{1}{4}=\frac{3}{4},\notag\\
q_i+p_i+r(q_i-\bar{x}_i)>\frac{3}{4}-\frac{1}{2}>0 &\implies& q+p+r(q-\bar{x}) \in \mathbb{R}_+^n.\notag
\end{eqnarray}
Finally,
$$0 \in q+p+r(q-\bar{x})+\mathbb{R}_-^n=p+q+r(q-\bar{x})+N_{K(q)}(q) \implies q \in \mathbf{M^r}(p,q).$$
Strong monotonicity of $x\mapsto x+p+r(x-\bar{x})$ implies that the solution is unique. Thus, $q=\mathbf{M^r}(p,q)$, $L'_p=0,L'_q=1$ and Assumption~\ref{calmr} is satisfied.
\end{example}

We are now ready to state
 the localization assumption required in \cite{evens2025}, given in \eqref{loc} below.

\begin{proposition}\label{propA4}
Under Assumption~\ref{calmr}, 
there exists an $\varepsilon > 0$ such that for every $(\mathbf{x},\mathbf{y}) \in \mathcal{W}_\varepsilon$, there exists a pair $(\mathbf{\tilde{x}},\mathbf{\tilde{y}})$ such that
\begin{equation}\label{loc}(\mathbf{\tilde{x}}+\frac{1}{r}(\mathbf{y}-\mathbf{\tilde{y}}),\mathbf{\tilde{y}}+r(\mathbf{x}-\mathbf{\tilde{x}})) \in \graph \mathbf{T} \cap \mathbf{U},\end{equation}
where $\mathbf{U}$ is the semimonotonicity neighborhood as given in inequality~\eqref{ineqQ} and Theorem~\ref{A1AIII}.
\end{proposition}

\begin{proof}
Choose $\delta>0$ such that $\{(\mathbf{\bar{x}},\mathbf{0})\}+\mathbb{B}_\delta \subset \mathbf{U}$. Take $\delta_1,\delta_2>0$ such that $\mathbb{B}_{\delta_1}\subset V_p' \text{ and }\{\bar{x}\}+\mathbb{B}_{\delta_2}\subset V_q'$. First, take $\varepsilon>0$ such that 
\begin{equation}\label{ineqeps}
\varepsilon\leq \varepsilon_0=\min\left\{\delta_1\left(r\sqrt{\frac{\lambda_x}{2r}}+\sqrt{\frac{r\lambda_y}{2}}\:\right)^{-1}, \delta_2\left(\sqrt{\frac{\lambda_x}{2r}}+\frac{1}{r}\sqrt{\frac{r\lambda_y}{2}}\:\right)^{-1}\right\}.
\end{equation}
Given $(\mathbf{x},\mathbf{y})=\left(\begin{bmatrix}x\\x\end{bmatrix},\begin{bmatrix}y\\-y\end{bmatrix}\right) \in \mathcal{W}_\varepsilon,$ take $\tilde{y}_1=y, \tilde{y}_2=0, \tilde{x}_2=x,$ and let $\tilde{x}_1$ be the solution of the VI$(F+rI-rx-y,K(x-\frac{1}{r}y)).$ 
We have that
\begin{equation}\label{cond1}
(\mathbf{\tilde{y}}+r(\mathbf{x}-\mathbf{\tilde{x}}))_2=\tilde{y}_2+r(x-\tilde{x}_2)=0,
\end{equation}
and
\begin{equation}
0 \in F(\tilde{x}_1)+r\tilde{x}_1-rx-y+N_{K(x-\frac{1}{r}y)}(\tilde{x}_1) \iff y+r(x-\tilde{x}_1)
	\in F(\tilde{x}_1)+N_{K(x-\frac{1}{r}y)}(\tilde{x}_1). \label{cond21}
\end{equation}
Therefore,
\begin{equation}
(\mathbf{\tilde{y}}+r(\mathbf{x}-\mathbf{\tilde{x}}))_1\in F((\mathbf{\tilde{x}}+\frac{1}{r}(\mathbf{y}-\mathbf{\tilde{y}}))_1)+N_{K((\mathbf{\tilde{x}}+\frac{1}{r}(\mathbf{y}-\mathbf{\tilde{y}}))_2)}((\mathbf{\tilde{x}}+\frac{1}{r}(\mathbf{y}-\mathbf{\tilde{y}}))_1).\label{cond2}
\end{equation}
Conditions~\eqref{cond1} and~\eqref{cond2} are equivalent to
\[(\mathbf{\tilde{x}}+\frac{1}{r}(\mathbf{y}-\mathbf{\tilde{y}}),\mathbf{\tilde{y}}+r(\mathbf{x}-\mathbf{\tilde{x}})) \in \graph \mathbf{T}.\]
By~\eqref{cond21}, we have that
\begin{equation}\label{useMVIr}
0 \in F(\tilde{x}_1)+r(\tilde{x}_1-\bar{x})+r(\bar{x}-x)-y+N_{K(x-\frac{1}{r}y)}(\tilde{x}_1) \iff \tilde{x}_1=\mathbf{M^r}(r(\bar{x}-x)-y,x-\frac{1}{r}y).
\end{equation}
Using~\eqref{desiaux} and~\eqref{ineqeps}, we conclude that 
\begin{align}
\|r(x-\bar{x})-y\| \leq r\|x-\bar{x}\|+\|y\| \leq \varepsilon\left(r\sqrt{\frac{\lambda_x}{2r}}+\sqrt{\frac{r\lambda_y}{2}}\:\right)\leq \delta_1 \notag \\ \implies r(x-\bar{x})-y \in \mathbb{B}_{\delta_1}\subset V_p',\label{Bx}
\end{align}
and
\begin{align}
\|(x-\frac{1}{r}y)-\bar{x}\| \leq \|x-\bar{x}\|+\frac{1}{r}\|y\|\leq \varepsilon\left(\sqrt{\frac{\lambda_x}{2r}}+\frac{1}{r}\sqrt{\frac{r\lambda_y}{2}}\right)\leq \delta_2\notag \\ \implies x-\frac{1}{r}y \in \{\bar{x}\}+\mathbb{B}_{\delta_2} \subset V_q'.\label{By}
\end{align}
Assumption~\ref{calmr} applied to~\eqref{useMVIr} implies that 
\[\|\tilde{x}_1-\bar{x}\| \leq L_p'\|r(\bar{x}-x)-y\|+L_q'\|x-\frac{1}{r}y-\bar{x}\|.\]
Using the triangle inequality and~\eqref{desiaux} once again, we have
\[\|\tilde{x}_1-\bar{x}\|\leq (rL_p'+L_q')\|x-\bar{x}\|+(L_p'+\frac{1}{r}L_q')\|y\|\leq \varepsilon\left((rL_p'+L_q')\sqrt{\frac{\lambda_x}{2r}}+(L_p'+\frac{1}{r}L_q')\sqrt{\frac{r\lambda_y}{2}}\:\right).\]
Denote $K=(rL_p'+L_q')\sqrt{\frac{\lambda_x}{2r}}+(L_p'+\frac{1}{r}L_q')\sqrt{\frac{r\lambda_y}{2}},$
	so that 
\begin{equation}\label{desiaux2}
\|\tilde{x}_1-\bar{x}\|\leq K\varepsilon.
\end{equation}
We also have 
\begin{equation}\label{desiaux3}
\|x-\tilde{x}_1\| \leq \|x-\bar{x}\|+\|\bar{x}-\tilde{x}_1\| \leq \left(\sqrt{\frac{\lambda_x}{2r}}+K\right)\varepsilon.
\end{equation}
Note that 
\begin{eqnarray*}
	& &\|\mathbf{\tilde{x}}+\frac{1}{r}(\mathbf{y}-\mathbf{\tilde{y}})-
	\mathbf{\bar{x}}\|^2+\|\mathbf{\tilde{y}}+r(\mathbf{x}-\mathbf{\tilde{x}})-\mathbf{0}\|^2\notag \\ 
	&=&\left\|\begin{bmatrix}\tilde{x}_1\\x -\frac{1}{r}y\end{bmatrix}-\begin{bmatrix}\bar{x}\\
		\bar{x} \end{bmatrix}\right\|^2+\left\|\begin{bmatrix}y+r(x-\tilde{x}_1)\\0 \end{bmatrix}-\begin{bmatrix} 0\\0\end{bmatrix}\right\|^2 \notag\\
			&=&\|\tilde{x}_1-\bar{x}\|^2+\|x-\bar{x}-\frac{1}{r}y\|^2+\|y+r(x-\tilde{x}_1)\|^2 \notag \\ &\leq&
			\|\tilde{x}_1-\bar{x}\|^2+ 2(\|x-\bar{x}\|^2+\frac{1}{r^2}\|y\|^2)+2(\|y\|^2+r^2\|x-\tilde{x}_1\|^2)\notag \\ &\leq&
			(K\varepsilon)^2 + 2\left(\frac{\lambda_x}{2r}\varepsilon^2+\frac{1}{r^2}\frac{r\lambda_y}{2}\varepsilon^2\right)+2\left(\frac{r\lambda_y}{2}\varepsilon^2+r^2\left(K+\sqrt{\frac{\lambda_x}{2r}}\:\right)^2\varepsilon^2\right) \\
			&=&\alpha\varepsilon^2,\notag
\end{eqnarray*}
where
$$\alpha=K^2+\frac{\lambda_x}{r}+\frac{\lambda_y}{r}+r\lambda_y+2r^2\left(K+\sqrt{\frac{\lambda_x}{2r}}\right)^2.$$
Choosing $\varepsilon\leq \min\{\varepsilon_0,\frac{\delta}{\sqrt{\alpha}}\},$
we obtain that
\[\|(\mathbf{\tilde{x}}+\frac{1}{r}(\mathbf{y}-\mathbf{\tilde{y}}),\mathbf{\tilde{y}}+r(\mathbf{x}-\mathbf{\tilde{x}}))-(\mathbf{\bar{x}},\mathbf{0}))\|\leq \sqrt{\alpha}\varepsilon \leq \delta,\]
implying that
\[(\mathbf{\tilde{x}}+\frac{1}{r}(\mathbf{y}-\mathbf{\tilde{y}}),\mathbf{\tilde{y}}+r(\mathbf{x}-\mathbf{\tilde{x}}))\in \mathbf{U}.\]
\end{proof}

Our final result, stating local convergence, is shown for
monotone $F$, so that each VI subproblem has unique solution. This
follows from 
Theorem 2.3.3-(b) in~\cite{2004}, since the proximal term makes the operator strongly monotone. Monotonicity of $F$ is used only to ensure uniqueness 
of the iterates. In the absence of that property,
the same local argument would apply to any sequence generated,
by selecting a localized solution at each step, whose existence is guaranteed by Proposition~\ref{propA4}.

\begin{theorem}[Local convergence of Algorithm~\ref{algo2}]\label{loccon}
Assume that $\bar{x}$ is a solution of QVI~\ref{QVI2}, $F$ is monotone, Assumption~\ref{calmr} holds, 
	and the parameters in Algorithm~\ref{algo2} satisfy 
\eqref{parameters},
where $\mu,\mu^\perp>0$ with $\mu\mu^\perp<1$ are given by Theorem~\ref{A1AIII}. 
	Then the following statements hold whenever the starting point $(x^0,y^0)$ is close
	enough to $(\bar{x},0)$:
\begin{enumerate}[label=(\roman*)]
\item
If Assumption~\ref{m1} holds, then either a solution of QVI~\ref{QVI} is reached in a finite number of 
		iterations, or the sequence $\{x^k\}$ generated by Algorithm~\ref{algo2} is bounded and
		its every accumulation point is a solution of QVI~\ref{QVI}. In particular, 
		if $\bar{x}$ is the unique solution of QVI~\ref{QVI} in the neighborhood under 
		consideration, then the sequence converges to $\bar{x}$.
\item
If Assumption~\ref{m2} holds, then either the solution $\bar{x}$ is reached in a finite number of 
		iterations, or the generated sequence $\{x^k\}$ converges to the
		solution of QVI (unique in this case).
\end{enumerate}
\end{theorem}

\begin{proof}
Assumption III in \cite{evens2025} states the conditions required for local convergence of 
	the progressive decoupling+ algorithm, divided into four items: A1, A2, A3, A4. 
We proceed to show that these hold in our case.

Theorem~\ref{A1AIII} establishes the semimonotonicity property, which is 
		precisely the requirement in A1. 

	Item A2 requires outer semicontinuity of $\mathbf{T}$, which holds on the entire domain, as shown in the proof of Theorem~\ref{globcon}. 

	Item A3 requires
	the parameters to be chosen within nonempty sets specified above, with
 $\mu\mu^\perp < 1$. The latter condition is ensured by Theorem~\ref{A1AIII}.

Item A4 is precisely \eqref{loc}, which holds in our case by Proposition~\ref{propA4}.

Because of item A4, a localized solution exists at each iteration of Algorithm~\ref{algo2}. 
	Since monotonicity of $F$ ensures uniqueness of the VI solution at each step, 
	the iterates are uniquely defined and remain in the semimonotonicity neighborhood.

By item (iii) of Theorem 4.9 in \cite{evens2025}, whenever the initial point $(x^0,y^0)$ is 
	sufficiently close to $(\bar{x},0)$, equivalently, whenever $(\mathbf{x}^0,\mathbf{y}^0)$ is 
	sufficiently close to $(\mathbf{\bar{x}},\mathbf{0})$, the sequence
	$\{x^k\}$ generated by Algorithm~\ref{algo2} either reaches a solution in a finite number of 
	iterations or it is bounded and all its accumulation points are solutions of QVI~\ref{QVI}. 
	Then if $\bar{x}$ is the unique solution of QVI~\ref{QVI} in the neighborhood under consideration, 
	the sequence converges to $\bar{x}$. 
	The case of Assumption~\ref{m2} falls into this local uniqueness setting.
\end{proof}

\section{Numerical Results}\label{sec:num}
To assess computational performance of our proposal, we apply Algorithm~\ref{algo1} to solve 
Walrasian equilibrium problems formulated as GNEPs \cite{Arrow_1954}. 
A GNEP can be solved combining the optimality conditions of agents' problems,
which leads to a QVI, as  
in \cite{Facchinei_2007, Facchinei_2013}. We compare our method with a direct solution using the Extended Mathematical Programming (EMP) module in the software GAMS \cite{GAMS},
 described in \cite{Kim_2019}. 
The latter method uses safeguarded semismooth Newton techniques to solve 
the mixed complementarity reformulation of GNEP, using the PATH solver \cite{Dirkse_1995}. 
We also benchmark Algorithm~\ref{algo1} against the recently proposed QVI Dantzig-Wolfe decomposition method \cite{jss25}.

Before solving the large-scale Walrasian equilibrium problems, we 
tune the proximal parameter on a battery of QVI problems available in the literature.
These experiments were conducted to understand the sensitivity of the method 
with respect to the proximal parameter $r>0,$ 
to gain some intuition on what values are suitable for computational purposes. 
To this end, we apply the method to the academic problems from \cite{QVILIB} and to other Walrasian equilibrium problems from \cite{Deride_2017}. 

All tests were conducted on a computer running Ubuntu 22.04 with an AMD Ryzen Threadripper 1950X processor featuring 16 cores (32 threads) and 64 GB of RAM. The codes are available at \url{https://github.com/ManoelJardim/ProgQVI/tree/main/ProgQVI}. We used GAMS with its default parameter settings. All codes were written in Python, with explicit calls to GAMS and other necessary packages when required.

\subsection{Proximal Parameter Analysis}\label{sec:sensr}
In \cite{Deride_2017}, numerical examples of the Arrow-Debreu model 
with utility functions based on the constant elasticity of substitution
are presented. We simulated Examples 3.7, 3.8, and 3.9 from that paper. The remaining QVI problems solved in this section are described in \cite{QVILIB}, where the authors propose a benchmark library of QVI problems called QVILIB. Since the realistic Walrasian problems are already part of our test set, we selected the academic problems from QVILIB. 

The experiment includes
all 29 academic problems from QVILIB and the 3 examples from \cite{Deride_2017}, amounting to 32 problems.

For each problem, we tested the proximal parameter for the fixed values $r \in \{0.001, 1, 100\}$,
and monotonically increasing parameters 
given by  $r_k = \min(0.001\times 2^k, 10^6)$,
starting at $0.001$ and bounded above by $10^6$. 
The initial point was set to zero, $x^0 = y^0 = 0_{\mathbb{R}^n}$,  except for four QVILIB instances for which $0_{\mathbb{R}^n}$ is already a solution.  In these cases, we set the initial point to  $x^0 = \begin{bmatrix} 1 \\ \vdots \\ 1 \end{bmatrix}$ and  $y^0 = 0_{\mathbb{R}^n}$. We counted the number of iterations until the stopping criterion measure fell below $10^{-3}.$ 
	The performance profile in Figure~\ref{fig:performanceprofile} indicates the sensitivity of Algorithm~\ref{algo1} to the choice of the proximal parameter. The performance factor was defined in terms of the number of iterations, which was set to infinity for cases in which the algorithm diverged.

\begin{figure}[htbp]
     \centering 
      \includegraphics[width=0.9\textwidth]{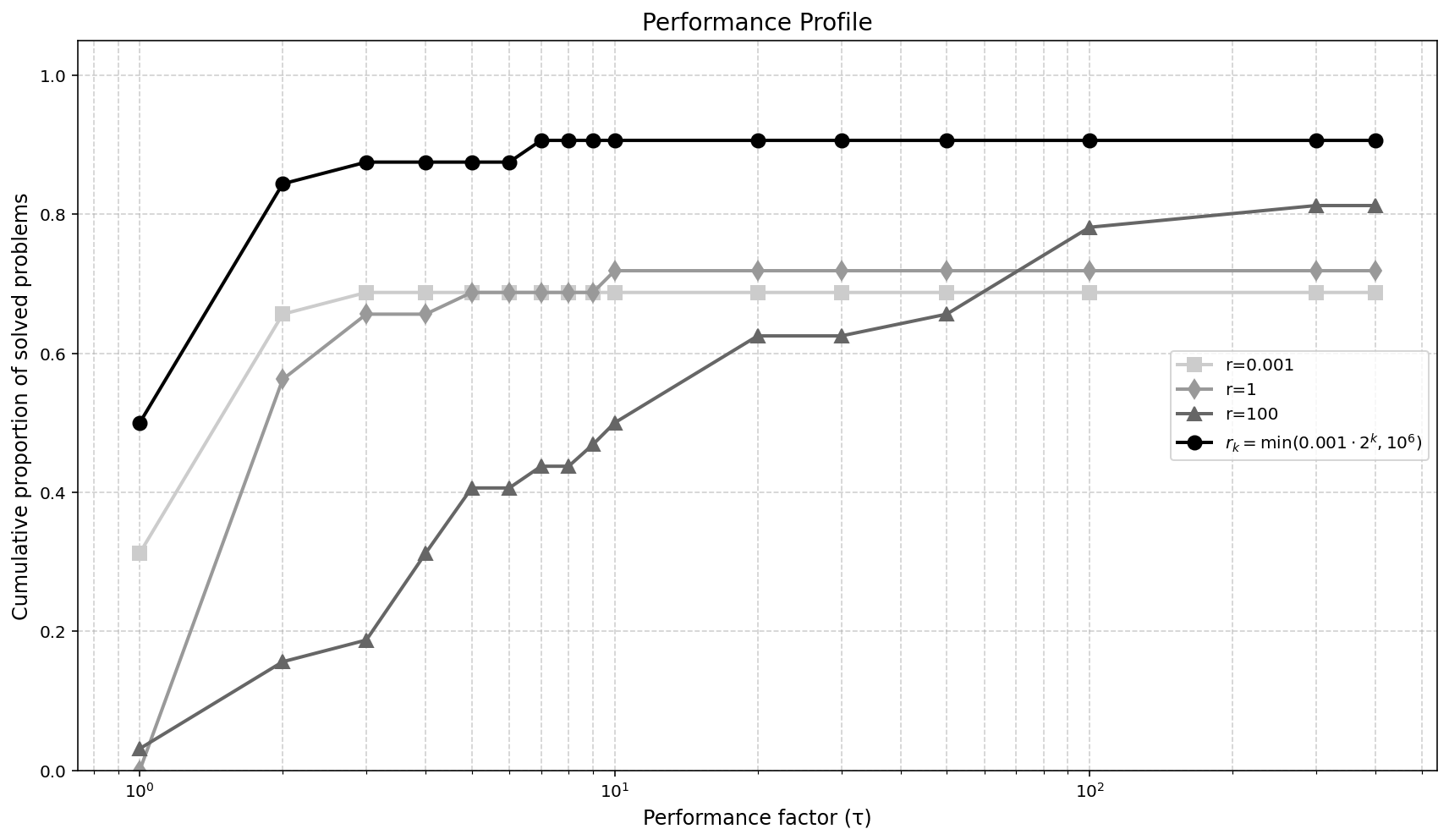}
      \caption{Performance profile comparing different proximal parameters.}
     \label{fig:performanceprofile}	
\end{figure}

Convergence was observed for almost all problems for at least one choice of the proximal parameter or by using the increasing parameter. Only the Box1B problem from QVILIB did not converge for any of the parameters tested. This problem is known to be difficult and has been reported in the literature \cite{Kanzow_2016,Kanzow_2017}, with most methods failing to converge except for exact penalization techniques in \cite{Kanzow_2017}. Algorithm~\ref{algo1} is based on solving a variational inequality (VI) subproblem at each iteration; therefore, its convergence depends on the robustness of the subproblem solver. This explains the divergence observed for the Box1B problem, which poses challenges to VI solvers in its subproblems.

Smaller values of the proximal parameter led to faster convergence. On the other hand, some problems required larger values of $r$ to converge, which could be mitigated by using a variable proximal parameter. The increasing parameter strategy achieved convergence in fewer iterations in most cases, combining the faster convergence of small proximal parameter values with the stability obtained for larger values.

\subsection{Walrasian Equilibrium Problem}
We closely follow the Walrasian problem formulation in \cite{jss25}. There are $C$ consumers, $G$ goods, one firm and one regulator doing the market clearance, totaling $C+2$ players with $G$ goods.
\subsubsection{VI subproblems}
For $i=1,...,C$, denote $x^i\in \mathbb{R}^G$ the goods to be bought by consumer $i$. To avoid confusion, in this section, iteration indices will be written in parentheses, such as $x^{(k)}$ or $x^{i,(k)}.$    Each consumer has a concave utility function $\mathcal{U}^i(x^i)$ and a budget constraint defined by its endowment $\mathcal{E}^i \in \mathbb{R}^G$: 
\begin{equation}
\max_{x^i\geq 0} \left\{\mathcal{U}^i(x^i): \left\langle p,x^i \right\rangle \leq \left\langle p,\mathcal{E}^i \right\rangle \right\}.
\end{equation}
The firm produces a quantity of goods, denoted by $x^{C+1}\in \mathbb{R}^G$, to maximize its profit under a production constraint, but it could be absent in a purely exchange economy, having the problem
\begin{equation}
\max_{x^{C+1}\geq 0}\left\{\left\langle p, x^{C+1} \right\rangle: \sum_{j=1}^G (x_j^{C+1})^2 \leq M \right\}.
\end{equation}

The price of each good $p^i\in\mathbb{R}^G$ is set by the market player to achieve market clearance, when there is an equilibrium between demand and supply, equivalent to Walras's law,
\begin{equation}
\max_{p\geq 0} \left\{ \left\langle p, \sum_{i=1}^C (x^i-\mathcal{E}^i)-x^{C+1}\right\rangle: \sum_{j=1}^G p_j=1 \right\}.
\end{equation}
Labeling the market player with $C+2$ so that $p=x^{C+2}$, $x=\begin{bmatrix}x^1 \\ \vdots \\x^{C+2}\end{bmatrix},$ the GNEP can be solved by the following QVI$(F,K)$:
\begin{align}
&F(x)=\Bigl(
-\nabla_{x^1}\mathcal{U}^1(x^1),\ldots,-\nabla_{x^C}\mathcal{U}^C(x^C),
-x^{C+2},
\sum\limits_{i=1}^C (\mathcal{E}^i-x^i)+x^{C+1}
\Bigr),
\label{Fqvi}\\
&K(x)=\{z\in \mathcal{D} \subset \mathbb{R}^{G(C+2)}: \sum\limits_{j=1}^Gx_j^{C+2}(z_j^i-\mathcal{E}_j^i)\leq 0, \text{ for } 1\leq i \leq C\}, \label{Kqvi} \\ 
&\mathcal{D}=\{z \in \mathbb{R}^{G(C+2)}:z\geq 0, \sum\limits_{j=1}^G (z_j^{C+1})^2 \leq M, \sum\limits_{j=1}^Gz_j^{C+2}=1\}.\label{setD}
\end{align}

The main cost of the iteration in our algorithm consists in solving a VI. This VI can be decoupled 
employing the structure of the convex set $K^{(k)}.$ We next write explicitly the iterations
VI for the mappings \eqref{Fqvi},\eqref{Kqvi}, simplifying the notation by
renaming the vectors for the calculations: 
$x=\hat{x}^{(k)}, d=y^{(k)}+rx^{(k)}$ and $q=x^{(k)}-\frac{1}{r}y^{(k)}$. This gives the following. 
\[                            
\text{Find } x \in K^{(k)} \text{ s.t. }\langle F(x)+rx-d,z-x\rangle \geq 0, \;\;
\forall z \in K^{(k)},\]
which means    
\begin{eqnarray}\label{VId}
	\text{Find }x\in K^{(k)} \text{ s.t. }& &
\sum\limits_{i=1}^C\langle -\nabla_{x^i}\mathcal{U}^i(x^i)+rx^i-d^i, z^i-x^i\rangle\notag \\ 
	&+& \langle -x^{C+2}+rx^{C+1}-d^{C+1}, z^{C+1}-x^{C+1} \rangle \notag \\
	&+& \langle \sum\limits_{i=1}^C 
	(\mathcal{E}^i-x^i)+x^{C+1}+rx^{C+2}-d^{C+2}, z^{C+2}-x^{C+2}\rangle \geq 0, \notag \\
	&& \;\;\;\forall z \in K^{(k)}.
\end{eqnarray}
Note that
\begin{align}
&K^{(k)}=\{z \in \mathcal{D}: \sum\limits_{j=1}^G q_j^{C+2}(z_j^i-\mathcal{E}_j^i) \leq 0, \text{for } 1\leq i \leq C\}= \prod\limits_{i=1}^C D^{i} \times D^{C+1} \times D^{C+2}, \notag \\ 
&\text{where } D^i=\{z^i\in \mathbb{R}_+^G:\sum\limits_{j=1}^G q_j^{C+2}(z_j^i-\mathcal{E}_j^i) \leq 0\}, \notag \\ 
&D^{C+1}=\{z^{C+1} \in \mathbb{R}_+^G: \sum\limits_{j=1}^G(z_j^{C+1})^2 \leq M\}, D^{C+2}=\{z^{C+2} \in \mathbb{R}_+^G: \sum\limits_{j=1}^G z_j^{C+2} = 1\}.\notag
\end{align}
For each $i \in \{1,...,C\},$ take $z \in K^{(k)},$ with $z^m=x^m \text{ when } m\neq i, z^i\in D^i$ 
arbitrary, $z^{C+1}=x^{C+1} \text{ and } z^{C+2}=x^{C+2}$.
Applying this in \eqref{VId}, we get VIs in a small-dimensional space $\mathbb{R}^G$:
\begin{equation}\label{VIi}
	\mbox{Find } x^i \in D^i \text{ s.t. } 
	\langle -\nabla_{x^i}\mathcal{U}^i(x^i)+rx^i-d^i, z^i-x^i\rangle \geq 0, \; \forall z^i \in D^i.
\end{equation}
Taking $z \in K^{(k)}$, with $z^i=x^i \text{ for } 1\leq i\leq C, z^{C+1}\in D^{C+1}, z^{C+2}\in D^{C+2},$
applying these in \eqref{VId}, we get the VI          
\begin{align}
	\mbox{Find }
	\begin{bmatrix} x^{C+1} \\ x^{C+2} \end{bmatrix} \in D^{C+1}\times D^{C+2}\;\;\mbox{ s.t. }
		\qquad \qquad \qquad \qquad \qquad \qquad \qquad \qquad \qquad \notag \\ \left\langle \begin{bmatrix} -x^{C+2}+rx^{C+1}-d^{C+1} \\ \sum\limits_{i=1}^C (\mathcal{E}^i-x^i) +x^{C+1}+rx^{C+2}-d^{C+2}\end{bmatrix}, \begin{bmatrix} z^{C+1}\\z^{C+2}\end{bmatrix}-\begin{bmatrix} x^{C+1} \\ x^{C+2} \end{bmatrix} \right\rangle \geq 0,\notag \\
\forall \begin{bmatrix} z^{C+1} \\ z^{C+2}\end{bmatrix} \in D^{C+1}\times D^{C+2}. \label{VIc12}
\end{align}

We still have the variables $x^i, i \in \{1,...,C\}$, in VI \eqref{VIc12}, 
but we can first solve the VIs in \eqref{VIi} to obtain these vectors $x^i$. 
Crucially, due to their structure, 
each of these VIs can be solved separately as a convex optimization problem. 
Once these vectors are determined, they are no longer variables in \eqref{VIc12}, 
which then becomes a VI on $\mathbb{R}^{2G}$. This decoupling allows us to solve at each iteration
$C$ separate 
$G$-dimensional VIs (as convex optimization problems!), followed by a single $2G$-dimensional VI, 
rather than tackling the large $G(C+2)$-dimensional VI.
This feature appears to be a great gain for large-scale problems.

\subsubsection{Computational Tests}

Our experiments compare the following three approaches:
\begin{enumerate}
\item Direct Method (DIRECT): using the EMP module from GAMS \cite{Kim_2019}, based on PATH solver. The parameters are used in default mode of GAMS;
\item Dantzig-Wolfe decomposition (DW): The DW decomposition for QVI proposed in \cite{jss25}. The VI subproblems are solved using convex programming, with CVXOPT package in Python for the quadratic case, and QVI master problems are solved using EMP. The initial point takes prices equal to $1/G$ and zero for the other values, while the stopping criterion requires to have {\sc gap} $\geq-0.01$;
\item Progressive decoupling of QVI (ProQVI): Algorithm \ref{algo1}. 
	Each VI decoupled in \eqref{VIi} solved using convex programming, 
		with CVXOPT package in Python for the quadratic case. 
		The nonsymmetric VI in \eqref{VIc12} is solved using GAMS. Based on the sensitivity analysis in section~\ref{sec:sensr}, and opting for a fixed proximal parameter, we use $r=0.001.$ The initial point sets the prices equal to $1/G,$ and assigns zero to the other values of $x^0$ and to $y^0$. The stopping criterion is based on $\tol = 0.001$.
\end{enumerate}

The data is generated randomly as done in \cite{jss25}. The utility functions are concave and quadratic $\mathcal{U}^i(x^i)=-\frac{1}{2}\langle x^i,\mathcal{R}^ix^i \rangle + \langle b^i, x^i \rangle,$ for $i=1,...,C$, where $b^i\in \mathbb{R}^G$ has elements uniformly distributed in $[0,10]$, and
$\mathcal{R}^i$ is a $G\times G$ positive semidefinite matrix generated by 
$\frac{10}{\|A^{iT}A^i\|_\infty}A^{iT}A^i,$ with $A^i \in \mathbb{R}^{G\times G}$ having
elements uniformly distributed between $[-1,1]$, so that the elements of $\mathcal{R}^i$ are 
between $[-10,10],$ regardless of the size of $G$.  
The endowments are also obtained randomly with uniform distribution between $[0,10]$ and the firm capacity is proportional to $CG$, maintained at levels sufficient to meet market demand; however, this parameter has minor influence on the numerical results.

\begin{table}[htbp]
\caption{CPU time for solvers {\sc direct}, {\sc dw} and ProQVI, and number of {\sc dw} and ProgQVI iterations with
600 instances of the Walrasian equilibrium problem. 
In each row, the solver in bold face is the one having the lowest mean time of execution. The three cases marked  with $*$ had 1 divergence in 20 simulations, while the two marked with $\#$ had 2 divergences between 20 simulations.}\label{tab1}%
\begin{tabular}{@{}cccccccccc@{}}
\toprule
\multirow{2}{*}{\mbox{$(C,G,n)$}}&
\multicolumn{2}{l}{\mbox{{\sc direct} time (s)}} &
\multicolumn{2}{l}{\mbox{{\sc dw} time (s)}} &
\multicolumn{2}{l}{\mbox{{\sc prog} time (s)}} &
\multicolumn{2}{l}{Iterations (mean)} 
\vspace{.5em}\\
\multicolumn{1}{l}{}& mean & max & mean & max & mean & max & \sc dw & \sc prog\\
\midrule
(20, 20, 440) &\bf 0.2 & 0.2 & 2.3 & 4.1  & 2.1 & 4.5 &16.2&33.6 \\
(40, 40, 1,680) &\bf  0.8 & 1.1 & 3.8 & 7.3 & 2.9$^*$ & 5.9& 18.8 & 39.5\\
(60, 60, 3,720) & 6.4 & 74.1 & 4.3 & 8.4 & \bf 4.2 & 6.4&18.2 & 58.5\\
(80, 80, 6,560) & 23.5&220.4 & \bf 4.8 & 7.0 & 6.2 & 8.6& 15.9 & 80.6\\
(100, 100, 10,200) &256.9 &1,710.5 & 4.7 & 8.7 & \bf 4.1  & 10.8 & 13.0& 54.0\\
(120, 120, 14,640) &702.9 &5,887.3 & 5.8 & 9.7 & \bf 4.0 & 4.7 &12.7& 48.8\\
(140, 140, 19,880) & 583.8 &3,855.4 & 5.8 & 9.6 & \bf 4.5 & 5.7 & 10.5& 54.4\\
(160, 160, 25,920) & 1,660.3 &13,803.8 &6.5 &10.9 & \bf5.8 & 6.6&9.1 &60.5 \\
(180, 180, 32,760) & 1,867.2&4,568.6 &7.4 &13.9 & \bf6.7 & 9.8& 8.2&60.5 \\
(200, 200, 40,400) & 1,844.7&10,596.5 & \bf7.7& 17.9& 7.8 &9.1 &6.8 &73.1 \\
\midrule
(10, 20, 240) &\bf 0.1 & 0.1& 1.3& 3.0&1.9 &2.2& 10.4& 29 \\
(20, 40, 880) &\bf 0.4 & 0.4& 1.8& 2.8&2.3 &3.0& 13.0& 35.7 \\
(30, 60, 1,920) &\bf 1.2 & 1.3& 2.3& 3.9&3.9 &5.1& 13.8&54.9   \\
(40, 80, 3,360) &2.9 & 3.1& \bf 2.3& 4.6&4.6 &6.5& 11.9&60.4  \\
(50, 100, 5,200) &6.0 & 7.6& \bf 2.1& 4.5&5.3 &8.2& 10.3& 67.6 \\
(60, 120, 7,440) &60.8 & 915.8& \bf 2.5& 3.6&4.8 &9.9& 9.6&54.9 \\
(70, 140, 10,080) &33.9 & 65.1& \bf 2.5& 4.5&3.5 &4.1& 7.5&41.7 \\
(80, 160, 13,120) &139.1 & 1,731.9& \bf 3.6& 6.0&4.6 &5.3& 8.5&49.4 \\
(90, 180, 16,560) &77.5 & 213.7& \bf 3.8& 6.0&5.4 &6.3& 7.2&54 \\
(100, 200, 20,400) &88.6 & 333.1& \bf 3.1& 6.0&5.5 &7.0& 5.3&55.3 \\ \midrule
(20, 10, 220) &\bf 0.1 &0.1 &3.1 &5.1 &1.8$^*$ &2.0& 19.8& 29.5\\
(40, 20, 840) &\bf 0.3 &0.6 &4.8 &8.2 &2.0$^*$ &2.5& 24.4& 32.9 \\
(60, 30, 1,860) &\bf 0.8 &1.3 &8.3 &16.1 &2.6 &3.8& 28.9& 40.0\\
(80, 40, 3,280) &\bf 2.6 &14.2 &12.2 &24.0 &3.4 &5.1& 31.3& 50.7 \\
(100, 50, 5,100) &15.8 &57.5 &15.0 &41.9 &\bf4.5 &6.4& 30.5& 66.8 \\
(120, 60, 7,320) &69.9 &195.3 &15.6 &33.2 &\bf5.1 & 7.6& 29.0& 72.9 \\
(140, 70, 9,940) &169.3 &345.9 &17.4 &87.8 &\bf4.5$^\#$ &8.5& 27.0& 62.8 \\
(160, 80, 12,960) &343.2 &868.5 &16.5 &49.3 &\bf3.3$^\#$ & 7.3& 26.0& 53.7  \\
(180, 90, 16,380) &699.8 &2,741.0 &14.4 &25.3 &\bf3.7 & 7.8& 22.0 & 52.0 \\
(200, 100, 20,200) &1,122.3 &3,836.6 &16.3 &45.8 &\bf4.1 & 4.8& 21.5& 54.3 \\
\bottomrule
\end{tabular}
\end{table}

The problem can increase in dimension by having more consumers $C$ or having more goods $G$, or even increasing both simultaneously. While we have $G(C+2)$ variables, we have $C$ main constraints in the definition of $K(x)$ in \eqref{Kqvi}, besides the constraints in the definition of the set $\mathcal{D}$. For our analysis, we increase $G$ and $C$ simultaneously, in three different benchmarks, based on the relationships $C/G \in \{0.5, 1, 2\},$ allowing to evaluate big problems with different structure concerning the dimensions of the variables and constraints.

The results in Table \ref{tab1} show that both decomposition methods behave very well for large scale problems, surpassing the direct approach for these instances. For small problems, the direct approach is advantageous, since the problem is solved almost instantaneously. This was already observed when
testing the DW approach \cite{jss25}. For the ProQVI, we observe a very good performance, even surpassing the DW decomposition for several cases. The reason is that after decoupling, 
all the subproblems in ProQVI decomposition are 
simple VIs (some even convex optimization problems), not needing any QVI solution
(like DW decomposition). Particularly, for problems with more constraints, 
the ProQVI approach is naturally superior, 
since the DW technique needs to solve its QVI master subproblems. 
However, for problems with fewer constraints, DW was superior to ProQVI, although the execution time
was similar. In conclusion, both ProQVI and DW decomposition are useful for
large problems, with the number of constraints in the problem determining which one of the two
approaches is superior.

The graphs in Figure~\ref{fig123} show the trend for the simulations, considering the statistical time, since we have solved 20 problems for each case, using the same starting points. As we can see, both decompositions present a small variation of time for the problems, and the ProQVI has the smallest dispersion, except for a few isolated instances which presented divergence (1.17\% of the cases for $r=0.001$). 

\begin{figure}[htbp]
    \centering
    \begin{subfigure}[b]{0.47\textwidth}
        \centering
        \includegraphics[width=\textwidth]{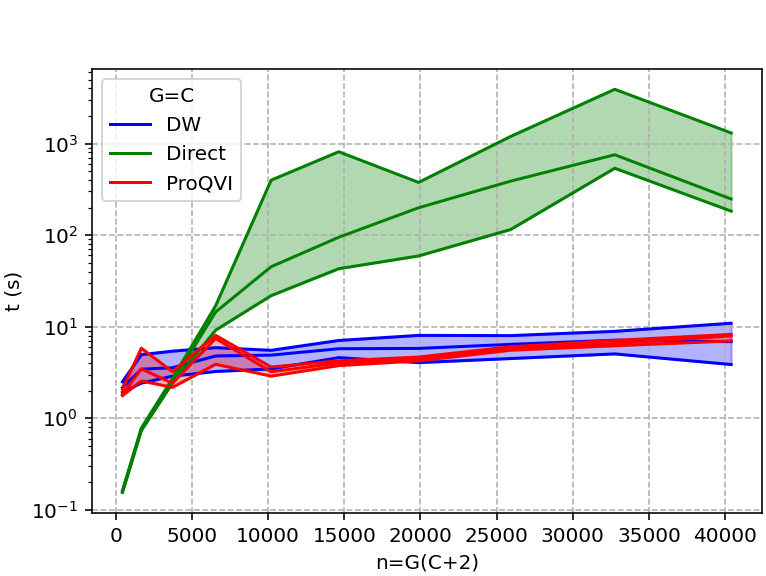}
    \end{subfigure}
    \hspace{0.0001\textwidth}
    \begin{subfigure}[b]{0.47\textwidth}
        \centering
        \includegraphics[width=\textwidth]{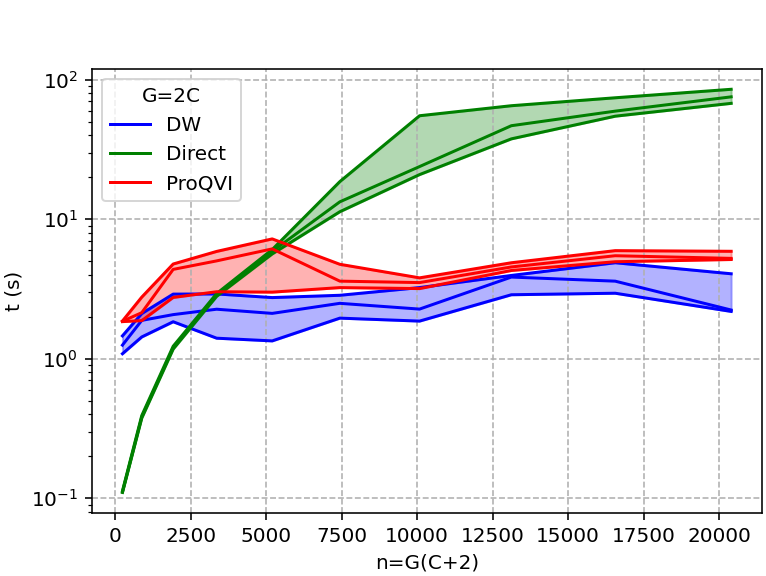}
    \end{subfigure}
    \vspace{0.3em}
    \begin{subfigure}[b]{0.55\textwidth}
        \centering
        \includegraphics[width=\textwidth]{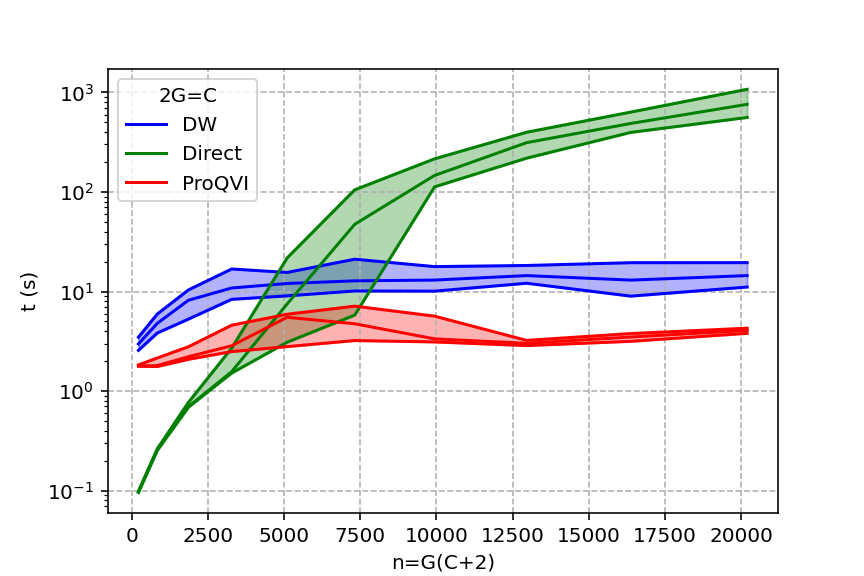}
    \end{subfigure}
  \caption{Time statistics (median and 25\%-75\% quantiles) 
for {\sc direct}, {\sc dw} and ProQVI, increasing $n$ with $G=C$, $G=2C$, and $2G=C$.}
    \label{fig123}
\end{figure}	




\section*{Declarations}

\noindent {\small
{\bf Conflicts of Interest.}
The authors declare that they have no conflict of interest of any kind related to the manuscript.}

\noindent {\small
{\bf Data Availability Statement.}
Data sharing is not applicable to this article.
}

\section*{Acknowledgements}

The second author is supported by CNPq Grant 307509/2023-0.
The third author is supported in part by CNPq Grant 306775/2023-9 and
by FAPERJ Grant E-26/200.165/2026.

The authors thank Professor Paulo Silva from UNICAMP for providing access to the
computational resources necessary for conducting our numerical experiments.

\bibliography{p}
\end{document}